\documentclass{amsart}

\usepackage[latin1]{inputenc}
\usepackage{amsmath}
\usepackage{amsfonts}
\usepackage{amssymb}
\usepackage[english]{babel}
\usepackage{multirow}
\usepackage[dvips]{graphicx}
\usepackage{color}
\usepackage{indentfirst}
\usepackage{dsfont}
\usepackage{fancyhdr}
\usepackage{setspace}
\usepackage{subfigure}
\usepackage{psfrag}
\usepackage{url}
\usepackage[multiple]{footmisc}
\usepackage{verbatim}

\newcommand{\Q}{\mathbb{Q}}
\newcommand{\Z}{\mathbb{Z}}
\newcommand{\F}{\mathbb{F}}

\newcommand{\p}{\mathfrak p}

\newcommand{\q}{\mathfrak q}

\newcommand{\OO}{\mathcal O}
\newcommand{\fp}{\mathfrak{p}}

\newtheorem{tm}{Theorem}[section]
\newtheorem{lemma}[tm]{Lemma}
\newtheorem{defin}[tm]{Definition}
\newtheorem{pp}[tm]{Proposition}

\newtheorem{rem}[tm]{Remark}
\newtheorem{conj}[tm]{Conjecture}
\newtheorem{cor}[tm]{Corollary}

\DeclareMathOperator{\Norm}{Norm}

\title{Recipes to Fermat-type equations of the form $x^r + y^r =Cz^p$}

\author{Nuno Freitas}
\thanks{This work was supported by the scholarship with reference $SFRH/BD/44283/2008$ from \textit{Funda\c{c}ao para a Ci\^{e}ncia e a Tecnologia, Portugal} and also by a grant from \textit{Fundaci\'o Ferran Sunyer i Balaguer}}

\date{}

\begin{document}

\begin{abstract}
We describe a strategy to attack infinitely many Fermat-type equations of signature $(r,r,p)$, where $r \geq 7$ is a fixed prime and $p$ is a prime allowed to vary. We use a variant of the modular method over some totally real subfields of $\Q(\zeta_r)$. In particular, to a solution $(a,b,c)$ of $x^r + y^r =Cz^p$ we will attach several Frey curves $E=E_{(a,b)}$. We prove modularity of all the Frey curves and the exsitence of a constant constant $M_r$, depending only on $r$, such that for all $p>M_r$ the representations $\bar{\rho}_{E,p}$ are absolutely irreducible. Along the way, we also prove modularity of certain elliptic curves that are semistable at all $v \mid 3$.\par
Finally, we illustrate our methods by proving arithmetic statements about equations of signature $(7,7,p)$. Among which we emphasize that, using a multi-Frey technique, we show there is some constant $M$ such that if $p > M$ then the equation $x^7 + y^7 = 3z^p$ has no non-trivial primitive solutions.
\end{abstract}

\maketitle

\section{Introduction}

Wiles' groundbreaking proof \cite{wiles} of the modularity of semistable elliptic curves over $\Q$ was also the final piece in the proof of Fermat's Last Theorem, which states that for all $n \geq 3$ there are no integers $a,b,c$ such that $a^n +b^n = c^n$ and $abc \neq 0$. Since then, the strategy that led to its proof was strengthened and several mathematicians achieved great success in solving other equations that previously seemed intractable. This generalized strategy is now known by \textit{the modular method} or \textit{the modular approach} to Diophantine equations. Broadly speaking, it can be divided into three main steps:
\begin{itemize}
\item[(I)][\textsc{Construction of a Frey curve}] Attach an appropriate elliptic curve $E$ (often called a Frey or Hellegouarch-Frey curve) defined over a totally real field $K$ to a putative solution (of a certain type) of a Diophantine equation;
\item[(II)] [\textsc{Modularity/Level Lowering}] Prove modularity of $E/K$ and irreducibility of some residual Galois representations $\bar{\rho}_E$ attached to $E$, to conclude (via level lowering results), that $\bar{\rho}_E$ corresponds to a (Hilbert) newform almost independent of the choice of the solution;
\item[(III)] [\textsc{Contradiction}] Contradict step (II) by showing that among the finitely many (Hilbert) newforms predicted in Step (II), none of them corresponds to $\bar{\rho}_E$.
\end{itemize}
Most of the efforts leading to the generalized modular method were made with the objective of studying \textit{the generalized Fermat equation} 
\begin{equation}
Ax^p + By^q = Cz^r, \quad \mbox{ where } \quad 1/p + 1/q + 1/r < 1,
\label{genfermat} 
\end{equation}
with $p,q,r \in \Z$ and $A,B,C$ pairwise coprime integers. To the triple of exponents $(p,q,r)$ as in (\ref{genfermat}) we call the \textit{signature} of the equation. In general, for fixed pairwise coprime integers $A,B,C$, equation (\ref{genfermat}) may have infinitely many solutions for a fixed signature. For example, if $z=a^3 + b^3$, $x=az$, $y=bz$ then $(x,y,z)$ satisfies $x^3 + y^3 = z^4$. However, if we assume the $abc$-conjecture it follows that there are only a finite number of solutions $(a,b,c)$ to equation (\ref{genfermat}) satisfying $\gcd(a,b,c)=1$ (see sections 1.1 and 5.2 in \cite{DG}). More precisely, equation \eqref{genfermat} is the subject of the following conjecture 
\begin{conj} Fix $A,B,C \in \mathbb{Z}$ pairwise coprime. There exist only finitely many triples $(a^p,b^q,c^r)$ with $(a,b,c)\in(\mathbb{Z}\setminus\{0\})^3$ and $p,q,r$ primes such that:
\begin{enumerate}
 \item $1/p+1/q+1/r<1$,
 \item $\gcd(a,b,c)=1$,
 \item $Aa^p+Bb^q=Cc^r$.
\end{enumerate}
\label{conjecture} 
\end{conj}
An important result due to Darmon-Granville \cite{DG} states that for $A,B,C$ fixed as above and a fixed signature $(p,q,r)$ such that $1/p+1/q+1/r<1$ there exists only a finite number of solutions satisfying $\gcd(a,b,c)=1$. Furthermore, the previous conjecture has been established in particular cases, including infinite subfamilies, for example: $x^ p + y^p = z^2$ for $p \geq 4$ and $x^ p + y^p = z^3$ for $p \geq 3$ were settled by Darmon-Merel \cite{DM} and Poonen \cite{Poonen}; $x^4 + y^2 = z^p$ for $p \geq 4$ by Ellenberg \cite{ell2} and Bennett-Ellenberg-Ng \cite{BENg}; $x^2 + y^6 = z^p$ for $p \geq 3$ by Bennett-Chen \cite{chenBen}. For an overview and up to date summary of known results see \cite{BCDY}.\\

We would like to note that all the generalizations of the modular approach used to attack different equations are highly dependent on the specific equation under analysis. There is no general algorithm that performs step (I) for a random Diophantine equation, even of Fermat type \eqref{genfermat}. However, a remarkable method to attack equation (\ref{genfermat}) in full generality is explained by Darmon in \cite{Darmon}. Darmon's method makes use of Frey abelian varieties of higher dimension, instead of the usual Frey curves. Unfortunately, it seems that currently little is known about these varieties and in \cite{Darmon} only a few particular cases of the equation $x^p + y^p = z^r$ are solved. 

\subsection{Summary of results}
In this work we focus on equations with signature $(r,r,p)$ and the particular form 
\begin{equation}
x^r + y^r = Cz^p,
\label{treze}
\end{equation}
where $C \in \Z$, $r \geq 7$ is a fixed prime and $p$ is a prime allowed to vary. Our main objective is to provide a strategy to attack infinitely many equations of this type for each fixed prime $r \geq 7$. More precisely, we will complete steps (I) and (II) of the modular approach for equations of the form \eqref{treze}, where $r \geq 7$, $C$ is an integer divisible only by primes $q \not\equiv 1,0$ (mod $r$). Thus, fixed $r$ and $C$, this reduces the problem of showing there are no solutions to \eqref{treze} to the computational part of the method, i.e. step (III). \\

Concerning the non-existence of solutions to equations with the shape of \eqref{treze} there are works for signature $(3,3,p)$ by Kraus \cite{kraus0}, Bruin \cite{bru}, Chen-Siksek \cite{CS} and Dahmen \cite{Dahm}; for $(5,5,p)$ by Billerey \cite{bil1}, Billerey-Dieulefait \cite{BD} and from the author jointly with Dieulefait \cite{DF}; for $(13,13,p)$ from the author jointly with Dieulefait \cite{DF2}. Furthermore, in \cite{kraus3} Kraus  proves that for each fixed pair $(r,p)$ of exponents, satisfying $r \geq 5$ and $p \geq 3$, there are only a finite number of values of $C$ (satisfying some natural conditions) for which \eqref{treze} admits solutions.\\

For the variant of the modular method used here we first introduce a preliminary step before steps (I), (II) and (III) above. Our method can be summarized as follows:
\begin{itemize}
\item[(0)] Relate a primitive solution $(a,b,c)$ of $x^r + y^r =Cz^p$ to a primitive solution of other Diophantine equations defined over $K^+$. The solutions of the new equations are of the form $(a,b,c_1)$ and will differ only on the value of $c_1 \in K^{+}$. 
\item[(I)] To solutions $(a,b,c_1)$ of each new equations we will attach Frey curves $E_{(a,b)}$ defined over $K^+$, obtaining this way multiple Frey curves attached to the initial solution $(a,b,c)$. 
\item[(II)] For each $r$ we prove modularity of all the $E=E_{(a,b)}$ and determine a constant $M_r$ such that the mod $p$ Galois representations $\bar{\rho}_{E,p}$ are absolutely irreducibility if $p > M_r$. Hence we can apply the level lowering results for Hilbert modular forms to get an isomorphism 
$$\bar{\rho}_{E,p} \sim \bar{\rho}_{f,\mathfrak{P}},$$ where $f$ is a Hilbert newform of level almost independent of $(a,b,c)$.
\item[(III)] Contradict the previous isomorphism for all predicted $f$. (Unfortunately, this cannot be done for general $r$ but we will exemplify it for particular cases when $r=7$.)
\end{itemize}

Despite being limited to equations with the form of \eqref{treze} the method given here has some advantages over the fully general method in \cite{Darmon}. Since we will only use elliptic curves, in some aspects, we are able to achieve a better understanding of the representations involved. More precisely, we are able to prove stronger statements about their modulariy and irreducibiliy. Also, in general, the mod $p$ Galois representations $\bar{\rho}_{E,p}$ arising from the $p$-torsion of our Frey curves do not fit the classification of Frey representations in \cite{Darmon}.

\subsection*{Notation} Fix $r \geq 7$ to be a prime. Denote by $K^+$ be the maximal totally real subfield of $\Q(\zeta_r)$; when $r \equiv 1 \pmod{6}$ write $K_0$ for the subfield of $K^+$ satisfying $[K^+ : K_0]=3$; when $r \equiv 1 \pmod{4}$ write $k$ for the subfield of $K^+$ satisfying $[K^+ : k]=2$. Let $\pi_r$ denote the single prime in $K^+$ dividing $r$; for an element $s \in \OO_{K^+}$ denote by $\mbox{Rad}_2(s)$ the product of the odd characteristic primes in $K^+$ dividing $s$.

\bigskip

We will now briefly describe our main results. Some of the statements are simplified for introduction purposes, but the chain of ideas leading to the completion of steps (I) and (II) should not be affected.

\bigskip

To complete step (I) we will construct several Frey curves as it is summarized in the following theorem. For the explicit definitions and properties of the Frey curves see sections \ref{curvasI}, \ref{curvasII} and \ref{4m}.

\begin{tm} Let $r \geq 7$ be a prime. Suppose that $(a,b,c) \in \Z^3$ satisfies \eqref{treze} with $a,b \in \Z$ coprime. Then, there are multiple Frey curves attached to $(a,b,c)$. More precisely,
\begin{enumerate}
 \item there are multiple Frey curves defined over $K^+$ for all $r\geq 7$,
 \item if $r \equiv 1 \pmod{6}$ there are Frey curves defined over $K_0$,
 \item if $r \equiv 1 \pmod{4}$ there are Frey curves defined over $K^+$ that are $k$-curves.
\end{enumerate}
\end{tm}

To complete step (II) we need modularity of the Frey curves and irreducibility of certain Galois representations attached to them. As an application of recent developments on modularity lifting theorems, we will prove, in particular, the following modularity statement. 

\begin{tm} Let $F$ be a totally real abelian number field where $3$ is unramified. Let $C/F$ be an elliptic curve semistable at all primes $v \mid 3$. Then, $C/F$ is modular.  
\end{tm}

We will also show that all the Frey curves we construct in this work are semistable at all $v \mid 3$. Thus, their modularity follows from the previous theorem. Concerning irreducibility we will prove the following.

\begin{tm} Let $r \geq 7$ be a fixed prime and let $(a,b,c)$ be a primitive solution $x^r + y^r =Cz^p$. There exists a computable constant $M_r$ (depending only on $r$) such that, for all $p > M_r$ the Galois representation $\bar{\rho}_{E,p}$ attached to any Frey curve $E$ is absolutely irreducible.
\end{tm}

At this point we completed steps (I) and (II). For a complete proof of the non-existence of certain solutions to equation \eqref{treze} we still have to perform step (III). Unfortunately, this step cannot be completed in general. In particular, for each value of $r$, it requires the computation of certain specific spaces of Hilbert newforms. Nevertheless, without having to perform any computation we are able to prove

\begin{tm} Let $C \neq 1,2$ be an integer divisible only by primes $q \not\equiv 1,0 \pmod{r}$. If certain explicit subspaces of $S_2(2^4 \pi_r^2 \mbox{Rad}_2(C))$ contain no newforms corresponding to elliptic curves with full $2$-torsion then, for $p$ greater than a constant, there are no solutions $(a,b,c)$ to the equation $x^r + y^r = Cz^p$, with $a,b$ coprime.
\end{tm}

We now introduce two natural definitions.

\begin{defin} Let $(a,b,c)$ be a solution to $x^r + y^r = Cz^p$. We will say that $(a,b,c)$ is a \textit{primitive} solution if $(a,b)=1$ and a \textit{trivial} solution if $|abc|\leq 1$.
\end{defin}

\begin{rem} Note that if $C$ is $r$-power free then $(a,b,c)$ being primitive in the sense above is equivalent to the condition gcd$(a,b,c)=1$ in Conjecture \ref{conjecture}. This latter condition is what usually is called primitive in the literature.
\end{rem}

Moreover, following the terminology introduced by Sophie Germain in her work on Fermat's Last Theorem, we will divide solutions into two cases.
\begin{defin} We will also say that $(a,b,c)$ is a first case solution if $r$ do not divide $c$, and a second case solution otherwise.
\label{sophie}
\end{defin}

Finally, we apply our methods to the study of equations of signature $(7,7,p)$. Among other things, we will prove the following theorems.

\begin{tm}
Let $d=2^{s_0}3^{s_1}5^{s_2}$ and $\gamma$ be an integer only divisible by primes $l \not\equiv 1,0 \mbox{ (mod } 7)$. Suppose $p \geq 17$. Then, the equation $x^7 + y^7 = d \gamma z^p$ has no non-trivial first case solutions if
one of the following three cases is satisfied:
\begin{enumerate}
 \item $s_0 \geq 2$, $s_1 \geq 0$ and $s_2 \geq0$;
 \item $s_0 = 1$, $s_1 \geq 1$ and $s_2 \geq 0$;
 \item $s_0 = 0$, $s_1 \geq 0$ and $s_2 \geq 1$.
\end{enumerate}
\label{gg7}
\end{tm}

\begin{tm} There is some constant $M_3$ such that if $p > (1 + 3^{18})^2$ and $p \nmid M_3$ then the equation $x^7 + y^7 = 3z^p$ has no non-trivial primitive solutions. 
\label{best}
\end{tm}

We want to emphasize that for the proof of the second theorem we need to use a multi-Frey technique with two families of Frey curves. One family is defined over $\Q$ and the other is defined over the totally real cubic subfield of $\Q(\zeta_7)$. The Frey curves over the cubic field are not $\Q$-curves, hence require the full strength of the methods developed here. Moreover, some of the computations involved in the proof are in the limit of what is possible. Indeed, only with a careful application of the multi-Frey technique we are able to reduce the amount of computations to the range of what is computable. All the hard computations in the proof were done by John Voight. 

\subsection*{Organization} In section~\ref{relacionar} we relate the equation $x^r + y^r = Cz^p$ to several other equations (i.e. step (0)); in each of sections~ \ref{curvasI}, \ref{curvasII} and \ref{4m} we construct multiple Frey curves of a certain type (step (I)) and prove, in particular, that they have the required ramification behavior. In each of this sections we also discuss the difficulties and advantages of using the corresponding type of curve in step (III); in section \ref{bgmod} we prove modularity of the Frey curves; in \ref{sec:irr} we prove irreducibility of the mod $p$ representations attached to the Frey curves, completing step (II). Finally, in sections~\ref{sete}, \ref{seteII}, \ref{seteIII} and \ref{seteIV} we apply the methods to study certain equations of signature $(7,7,p)$. 

\subsection*{Acknowledgments} My greatest thanks go to Luis Dieulefait for our numerous discussions. I also thank Samir Siksek and Panagiotis Tsaknias for their valuable suggestions. I am grateful to John Voight for his computions of Hilbert modular forms that were crucial to this work. I am indebted to Nicolas Billerey, Fred Diamond, Gabor Wiese, Sara Arias-de-Reyna and Xavier Guitart for helpful suggestions and comments. I also thank Gabor Wiese and Fred Diamond for having me as visitor at University of Luxembourg and King's College London, respectively. Great progress was made on this work during these visits. 

\section{Multiple Diophantine equations related to $x^r + y^r = Cz^p$.}
\label{relacionar}

Let $r \geq 7$ be a fixed prime and $C$ be an integer divisible only by primes $q \not\equiv 1$ (mod $r$). In this section we will relate primitive solutions of the equation 
\begin{equation}
\label{general}
x^r + y^r = Cz^p  
\end{equation}
with solutions of other equations defined over subfields of $\Q(\zeta_r)$. Write
$$\phi_r (x,y) = \frac{x^{r} + y^{r}}{x+y} = \sum_{i=0}^{r-1}{(-1)^{i}x^{r-1-i}y^i}.$$
The factorization 
$$x^{r} + y^{r} = (x+y)\phi_r (x,y)$$ 
plays a key role in this work, so we start by proving basic properties of $\phi_r$. 

\subsection{The factors of $\phi_r(x,y)$} Let $\zeta:=\zeta_r$ denote a primitive $r$-th root of unity. Over the cyclotomic field $\Q(\zeta)$ we have the factorization
\begin{equation}
\label{decomp}
\phi_r (x,y) = \prod_{i=1}^{r-1}{(x + \zeta^i y)}.
\end{equation}
\begin{pp} Let $\mathfrak{P}_r$ be the prime in $\Q(\zeta)$ above the rational prime $r$ and suppose that $(a,b)=1$. Then, any two factors $a + \zeta^i b$ and $a + \zeta^j b$ with $1 \leq j < i \leq r-1$ are coprime outside $\mathfrak{P}_r$. Furthermore, if $r \mid a+b$ then $\upsilon_{\mathfrak{P}_r}(a + \zeta^i b) = 1$ for all $1 \leq i \leq r-1$.
\label{factores}
\end{pp}
\begin{proof} Suppose that $(a,b)=1$. Let $\mathfrak{P}$ be a prime in $\Q(\zeta)$ above $p \in \Q$ and a common prime factor  of $a + \zeta^i b$ and $a + \zeta^j b$, with $i >j$. Observe that $(a + \zeta^i b) - (a + \zeta^j b) = b\zeta^j(1-\zeta^{i-j}) \in \mathfrak{P}$. Since $\mathfrak{P}$ cannot divide $b$ because in this case it would also divide $a$ we conclude that $\zeta^i(1-\zeta^{i-j}) \in \mathfrak{P}$ but $\zeta^{i}$ is a unit so $1-\zeta^{i-j} \in \mathfrak{P}$, that is $\mathfrak{P} = \mathfrak{P}_{r}$. Now for the last statement in the proposition, suppose that $r \mid a+b$. Then,
$$ a + \zeta^i b = a + b - b + \zeta^i b = (a+b) + (\zeta^i - 1)b,$$ 
and since $\upsilon_{\mathfrak{P}_r}(\zeta^i - 1) = 1$ we have $\upsilon_{\mathfrak{P}_r}(a + \zeta^i b) = \min\{(r-1)\upsilon_r(a+b),1\} = 1$. \end{proof}

\begin{cor} If $(a,b)=1$, then $a+b$ and $\phi_r (a,b)$ are coprime outside $r$. Furthermore, if $r \mid a+b$ then $\upsilon_{r}(\phi_r (a,b))=1$. 
\label{trezz}
\end{cor}
\begin{proof} Let $p$ be a prime dividing $a+b$ and $\phi_r (a,b)$ and denote by $\mathfrak{P}$ a prime in $\Q(\zeta)$ above $p$. Then $\mathfrak{P}$ must divide at least one of the factors $a + \zeta^i b$. Since $a,b$ are rational integers $\mathfrak{P}$ cannot divide $b$ then it follows from $$a+b = a + \zeta^i b - \zeta^i b + b = (a + \zeta^i b) + (1-\zeta^i)b$$ that $\mathfrak{P} = \mathfrak{P}_r$. Moreover, if $r \mid a+b$ it follows from the proposition that $\upsilon_{\mathfrak{P}_r}(a + \zeta^i b) = 1$ for all $i$ then $\upsilon_{\mathfrak{P}_r}(\phi_r(a,b)) = r-1$ thus $\upsilon_{r}(\phi_r (a,b)) = 1$.\end{proof}

\begin{pp} Let $(a,b)=1$ and $l \not\equiv 1$ (mod $r$) be a prime dividing $a^{r} + b^{r}$. Then $l \mid a+b$.
\label{trezz2}
\end{pp}
\begin{proof} Since $l$ divides $a^{r} + b^{r}$, $l \nmid ab$. Let $b_0$ be the inverse of $-b$ modulo $l$. We have $a^{r} \equiv (-b)^{r}$ (mod $l$), hence $(ab_0)^{r} \equiv 1$ (mod $l$). Thus the multiplicative order of $ab_0$ in $\mathbb{F}_l$ is 1 or $r$. From the congruence $ab_0 \equiv 1$ (mod $l$) it follows $a + b \equiv 0$ (mod $l$). If $l \nmid a+b$ then the order of $ab_0$ is $r$ and $l \equiv 1$ (mod $r$).\end{proof}

\subsection{Relating Diophantine equations} Recall that $C\not=0$ is an integer divisible only by primes satisfying $q \not\equiv 1\pmod{r}$. Assume further that $r \nmid C$. The following lemma relates solutions of the equation $x^r + y^r = Cz^p$ with solutions of two other equations over $\Q$.
\label{relating}
\begin{lemma}
Let $p$ be a prime and suppose there exists a primitive solution $(a,b,c')$ to $x^r+y^r=C z^p$. Then, there exists $c \in \mathbb{Z}$ such that $(a,b,c)$ is a solution to
\begin{equation}
\phi_r(a,b) = c^p \quad \mbox{or}
\label{casoA}  
\end{equation}
\begin{equation}
\phi_r(a,b) = rc^p
\label{casoB} 
\end{equation}
which satisfies $r\nmid a+b$ in case $(\ref{casoA})$ and $r\mid a+b$ in case $(\ref{casoB})$. Moreover,
\begin{itemize}
\item if $(a,b,c')$ is non-trivial then $|abc|>1$;
\item the prime divisors of $c$ are all congruent to $1\pmod{r}$. In particular, neither $2$, nor $r$ divide $c$.
\item there is $c_0 \in \Z$ coprime to $Cr$ such that $a+b = Cr^kc_0^p$ for some $k \geq 0$;
\end{itemize}
\label{novaeq2}
\end{lemma}
\begin{proof} Suppose there exists a non-trivial primitive solution $(a,b,c')$ to $x^r+y^r=C z^p$ and recall that $a^r + b^r = (a+b)\phi_r(a,b)$. Write $c' = p_1^{n_1}...p_s^{n_s} r^{m} q_1^{k_1}...q_l^{k_l}$, where $p_i, q_j$ are primes, $p_i \not\equiv 1,0 \pmod{r}$ for all $i$, $q_j \equiv 1 \pmod{r}$ for all $j$ and $m \in \Z_{\geq 0}$.\par 

If a prime $q \not\equiv 1 \pmod{r}$ divides $a^r + b^r$ then by Proposition \ref{trezz2} we have $q \mid a+b$; if $q \neq r$ divides $\phi_r(a,b)$ then $q \mid a^r + b^r$ and by Corollary \ref{trezz} $q \nmid a+b$. Hence by  Proposition \ref{trezz2} we have $q \equiv 1$ (mod $r$). Then, $Cc_0^p \mid a+b$, where $c_0 \mid c'$ is the product of all the $p_i^{n_i}$ with those $q_j^{k_j}$ not dividing $\phi(a,b)$.\par
Suppose $r \nmid a+b$, hence $m=0$ in the decomposition of $c'$. Since $a+b$ and $\phi_r(a,b)$ are coprime we have $a+b=Cc_0^p$ and $\phi_r(a,b) = c^p$, where $c' = c_0 c$. Suppose $r \mid a+b$, hence $m > 0$. Then, by Corollary \ref{trezz}, we have $a+b=Cc_0^pr^{mp-1}$, $\phi_r(a,b) = rc^p$ and $c' = c_0 c r^m$.\par 

We will now prove $|abc|>1$ if $(a,b,c')$ is non-trivial. Since $(a,b,c')$ is non-trivial we have $|ab| > 1$ or $|c'| > 1$. If $|ab| > 1$ then $(a,b,c)$ is non-trivial. If $|c'| > 1$ and $|c| > 1$ we are done; if $|c'| > 1$ but $|c| \leq 1$ we must have that $(c')^p \mid a+b$. This implies $a+b$ = 0 or $|a| > 1$ or $|b| > 1$.  If $a+b = 0$ then $(a,b) = \pm (1,-1)$ (because $(a,b)=1$) which implies $c' = 0$ contradicting $|c'| > 1$. Then $|a| > 1$ or $|b| > 1$ and we are done.\end{proof}

\begin{pp} Let $r \geq 5$ be a prime. The equations
$$\phi_r(x,y) = 1 \quad \quad \mbox{and} \quad \quad \phi_r(x,y) = r$$
admit only the solutions $\pm (1,0)$, $\pm (0,1)$, $\pm (1,1)$ and $\pm(1,-1)$, respectively.
\label{trivialsol}
\end{pp}
\begin{proof} Recall that $\phi_r (x,y)= \sum_{i=0}^{r-1}{(-1)^{i}x^{r-1-i}y^i}$ and suppose $(a,b)$ is a solution of any of the equations in the statement. From the symmetry of $\phi_r$ and the fact that $r-1-i$ and $i$ have the same parity we can suppose that we are in one of three possible cases: (i) $a=0$ or  (ii) $a > 0$ and $b < 0$ or (iii) $a \geq b > 0$.\par
Case (i): suppose $a=0$. Replacing on the equations, we have $b^{r-1} = 1$ or $b^{r-1} = r$. The first possibility gives the solutions $\pm (0,1)$ of the equation $\phi_r(x,y) = 1$ and by symmetry $\pm (1,0)$. The second is impossible because $r$ is prime $>2$.\par
Case (ii): suppose $a > 0$ and $b < 0$. Then $b=-b_0$, with $b_0$ positive. We see that $\phi_r (a,b)= \sum_{i=0}^{r-1}{a^{r-1-i}b_0^i} \geq r$, and it is clear that the equality holds only if $a=1$ and $b=-1$, corresponding to the solution $(1,-1)$ of the equation $\phi_r(x,y) = r$. We also have $(-1,1)$ by symmetry.\par
Case (iii): suppose $a \geq b > 0$. Note that $\phi_r(a,b)$ can be written in the form
$$\phi_r (a,b)= \sum_{i}{(a^{r-1-i}b^i - a^{r-1-(i+1)}b^{(i+1)})} + b^{r-1},$$
where the sum is over the even numbers $i$ satisfying $0 \leq i \leq r-3$. Suppose $a > b$ and observe that  
$$a^{r-1-i}b^i - a^{r-1-(i+1)}b^{(i+1)} = a^{r-1-i-1}b^{i}(a-b) \geq a^{r-1-i-1}b^{i} \geq 2^{r-1-i-1}.$$
Moreover, $2^{r-1-i-1} \geq 2$ and the equality holds only for $i=r-3$. Thus,
$$\phi_r (a,b)= \sum_{i}{(a^{r-1-i}b^i - a^{r-1-(i+1)}b^{(i+1)})} + b^{r-1} > (r-1)/2 \times 2 + 1 \geq r.$$
This shows that there are no Case (iii) solutions to both equations if $a > b$. Suppose $a=b$, then $\phi_r (a,b) = a^{r-1}$. This can be a solution of $\phi_r(x,y) = 1$ only if $a=b=1$. It can never be a solution of $\phi_r(x,y) = r$ because $r$ is prime $>2$.\end{proof}

\begin{cor} Let $(a,b,c)$ be a solution to equation $(\ref{casoA})$ or $(\ref{casoB})$. Then $|abc| > 1$ if and only if $|c| > 1$.
\label{trivialsol2} 
\end{cor}
\begin{proof} Suppose $|abc| > 1$. Then $|ab| > 1$ or $|c|>1$. If $|c|>1$ it is automatic. From Proposition \ref{trivialsol} we see that all solutions with $|c| = 1$ also satisfy $|ab| = 1$, hence, if $|ab| > 1$ we must have $|c|>1$. The other direction is immediate.\end{proof}

Recall that $K^{+}$ is the maximal totally real subfield of $\Q(\zeta)$ and let $h_r^{+}$ be its class number. Let also $\pi_r$ be such that $r\mathcal{O}_{K^{+}} = (\pi_r)^{(r-1)/2}$. Since $r \geq 7$ is a prime, $r-1 \geq 6$ is even $\phi_r$ factors over $K^+$ into degree two factors of the form 
$$f_k = x^2 + (\zeta^k + \zeta^{-k})xy + y^2, \quad \mbox{ where } \quad  1 \leq k \leq (r-1)/2.$$
We will now relate solutions of equations (\ref{casoA}) and (\ref{casoB}) to solutions of equations involving the quadratic factors $f_k$.

\begin{lemma} Let $p$ be a prime not dividing $h_r^{+}$. Suppose there is a primitive solution $(a,b,c')$ to $x^r+y^r=C z^p$. Then, for each $1 \leq k \leq (r-1)/2$ there exists a unit $\mu \in \mathcal{O}_{K^{+}}^{\times}$ and $c \in \mathcal{O}_{K^{+}}$ such that $(a,b,c)$ is a solution to
\begin{equation}
\label{eq11}
f_k(x,y) = \mu z^p \quad \mbox{ or}
\end{equation} 
\begin{equation}
\label{eq22}
f_k(x,y) = \mu \pi_r z^p,
\end{equation}
which satisfies $r\nmid a+b$ or $r\mid a+b$, respectively. Moreover:
\begin{itemize}
\item $|Norm_{K^{+}/\Q}(abc)| > 1$ if $(a,b,c')$ is non-trivial;
\item the primes in $K^{+}$ divisors of $c$ are all above primes of $\Q$ that are congruent to $1\pmod{r}$. In particular, neither the primes above $2$ nor the 
primes above $r$ divide $c$.
\item there is $c_0 \in \Z$ coprime to $Cr$ such that $a+b = Cr^kc_0^p$ for some $k \geq 0$;
\end{itemize}
\label{novaeq3} 
\end{lemma}
\begin{proof} Suppose that there is a non-trivial primitive solution $(a,b,c') \in \mathbb{Z}^3$ to $x^r + y^r = Cz^p$. Then, by Lemma \ref{novaeq2} there is a non-trivial primitive solution $(a,b,c'') \in \mathbb{Z}^3$ to equation (\ref{casoA}) or (\ref{casoB}). By Proposition \ref{factores} we know that the $f_k(a,b)$ are pairwise coprimes outside $\mathfrak{P}_r$. Then, since $\mathcal{O}_{K^{+}}$ is a Dedekind domain we have that: if $(a,b,c'')$ is a solution of (\ref{casoA}) or (\ref{casoB}) then $(f_k(a,b)) = \mathcal{I}^p$ or $(f_k(a,b)) = (\pi_r)\mathcal{I}^p$ as ideals in $\mathcal{O}_{K^{+}}$, respectively. These identities show that the order of $\mathcal{I}$ in the ideal class group of $K^{+}$ divides $p$. Since $p \nmid h_r^{+}$ we have that $\mathcal{I}$ is principal and we can write $f_k(a,b)= \mu c^p$ or $f_k(a,b)= \mu \pi_r c^p$, where $\mu, c \in \mathcal{O}_{K^{+}}$ with $\mu$ an unit.\par 
The two final conclusions follow trivially from the application of Lemma \ref{novaeq2} in this proof. We are left to show that $|Norm_{K^{+}/\Q}(abc)| > 1$ if $(a,b,c')$ is non-trivial. We have $|c''|>1$ by Corollary \ref{trivialsol2} and we have to show that $|Norm_{K^{+}/\Q}(abc)| > 1$. Let $q$ be a prime dividing $c''$, hence $q$ is congruent to $1 \pmod{r}$, i.e. $q$ splits in $\Q(\zeta)$. Since the factors $a + \zeta^i b$ of $\phi_r(a,b)$ are pairwise coprime outside $\mathfrak{P}_r$, each of them contains a non-trivial prime in $\Q(\zeta)$ above $q$. Thus, $c$ is divisible by some prime, hence $|Norm_{K^{+}/\Q}(abc)| > 1$.\end{proof}

\begin{defin}To a solution $(a,b,c)$ in $\mathbb{Z}^2 \times \mathcal{O}_{K^{+}}$ of an equation as in the previous lemma we will call \textit{primitive} if $(a,b)=1$ and \textit{non-trivial} if $|Norm_{K^{+}/\Q}(abc)| > 1$.
\end{defin}

\begin{cor} Let $(a,b,c) \in \mathbb{Z}^2 \times \mathcal{O}_{K^{+}}$ be a solution to $(\ref{eq11})$ or $(\ref{eq22})$. Then $(a,b,c)$ is non-trivial if and only if $|Norm_{K^{+}/\Q}(c)| > 1$.
\label{trivialsol3} 
\end{cor}
\begin{proof} Suppose $|Norm_{K^{+}/\Q}(abc)| > 1$. Then $|ab|>1$ or $|Norm_{K^{+}/\Q}(c)| > 1$. In the latter case we are done. Suppose $|ab|>1$, then by Corollary \ref{trivialsol2} there is a prime $q \mid \phi_r(a,b)$ different from $r$. Since the factors $a + \zeta^i b$ of $\phi_r(a,b)$ are pairwise coprime, each of them contains a non-trivial prime in $\Q(\zeta)$ above $q$. Thus, any product $(a + \zeta^i b)(a + \zeta^{r-i} b)$ is divisible by some prime, hence $|Norm_{K^{+}/\Q}(c)| > 1$. The other direction is immediate.\end{proof}

\begin{rem} From Corollaries $\ref{trivialsol2}$ and $\ref{trivialsol3}$ we see that a solution $(a,b,c)$ of $(\ref{casoA})$, $(\ref{casoB})$, $(\ref{eq11})$ or $(\ref{eq22})$ is non-trivial (i.e. $|Norm_{K^{+}/\Q}(abc)| > 1$) if and only if $c$ is divisible by some prime. As we will see later the primes dividing $c$ will correspond to multiplicative primes of the Frey curves. Moreover, we will see that they are the primes where level lowering can actually happen. We want to keep in mind this way of thinking about non-trivial solutions.
\end{rem}

\section{Constructing Frey curves part I}
\label{curvasI}

Let $r \geq 7$ be a prime. We will say that an integer triple $(k_1,k_2,k_3)$ is \textit{$r$-suitable} if $1\leq k_1 < k_2 < k_3 \leq (r-1)/2$. Notice in particular that any triple $(k_1,k_2,k_3)$ of strictly increasing integers is $r$-suitable if and only if $k_3\leq(r-1)/2$. Fix an $r$-suitable triple $(k_1,k_2,k_3)$ and consider the three degree 2 polynomials attached to it
$$\begin{cases}
 f_{k_1}(x,y) = x^2 + (\zeta^{k_1} + \zeta^{-k_1})xy + y^2, \\
 f_{k_2}(x,y) = x^2 + (\zeta^{k_2} + \zeta^{-k_2})xy + y^2,\\
 f_{k_3}(x,y) = x^2 + (\zeta^{k_3} + \zeta^{-k_3})xy + y^2.\\
\end{cases}$$
Suppose there is a primitive solution $(a,b,c')$ to $x^r + y^r = Cz^p$. From Lemma \ref{novaeq3} there are $\mu_i, c_i$ such that $f_{k_i}(a,b)= \mu_i c_i^p$ or
$f_{k_i}(a,b)= \mu_i \pi_r c_i^p$ for $i=1,2,3$. Hence, by multiplying the $f_{k_i}$ , there are $\mu, c \in \OO_{K^+}$, with $\mu$ a unit, such that $(a,b,c) \in \Z^2 \times \OO_{K^+}$ is a solution to
\begin{equation}
\label{eq1}
f_{k_1}(x,y)f_{k_2}(x,y)f_{k_3}(x,y) = \mu z^p \quad \mbox{ or}
\end{equation} 
\begin{equation}
\label{eq2}
f_{k_1}(x,y)f_{k_2}(x,y)f_{k_3}(x,y) = \mu \pi_r^3 z^p,
\end{equation}
which satisfies $r\nmid a+b$ or $r\mid a+b$, respectively. Unless stated otherwise, every time we will refer to equations (\ref{eq1}) and (\ref{eq2}) we are considering them with the $f_{k_i}$ corresponding to the fixed triple $(k_1,k_2,k_3)$. Moreover, when we talk of solutions $(a,b,c)$ to equation (\ref{eq1}) or (\ref{eq2}) we will be referring to solutions in the set $\mathbb{Z}^2 \times \mathcal{O}_{K^{+}}$. For equations (\ref{eq1}) and (\ref{eq2}) we will continue to call a solution \textit{non-trivial} if and only if $c$ is divisible by some prime.\\

We want to attach an elliptic curve to a putative primitive solution $(a,b,c)$ of equations (\ref{eq1}) or (\ref{eq2}). For that, we are interested in finding a triple $(\alpha, \beta, \gamma)$ such that
$$\alpha f_{k_1} + \beta f_{k_2} + \gamma f_{k_3} = 0,$$
which amounts to solve a linear system in the coefficients of the $f_{k_i}$. From the form of the $f_{k_i}$ we see this is always possible, because the linear system has two equations and three variables. In particular, we choose the solution
$$\begin{cases}
 \alpha = \zeta^{k_3} + \zeta^{-k_3} - \zeta^{k_2} - \zeta^{-k_2},\\
 \beta  = \zeta^{k_1} + \zeta^{-k_1} - \zeta^{k_3} - \zeta^{-k_3}, \\
 \gamma = \zeta^{k_2} + \zeta^{-k_2} - \zeta^{k_1} - \zeta^{-k_1}.\\
\end{cases}$$
Finally, given a primitive solution $(a,b,c)$ to equation (\ref{eq1}) or (\ref{eq2}) we put 
$$A_{(a,b)} = \alpha f_{k_1}(a,b), \quad \quad B_{(a,b)} = \beta f_{k_2}(a,b), \quad \quad C_{(a,b)} = \gamma f_{k_3}(a,b),$$ 
and we attach to $(a,b,c)$ the Frey curves over $K^{+}$ of the form
\begin{equation}
E^{(k_1,k_2,k_3)}_{(a,b)} : Y^2 = X(X-A_{(a,b)})(X+B_{(a,b)}).
\label{Freycurve}
\end{equation}

Observe that, fixed $r$, by changing the triple $(k_1,k_2,k_3)$ in \eqref{Freycurve} we obtain $\binom{(r-1)/2}{3}$ different families of Frey curves. For each $r$, most of the results in this section are independent of $(k_1,k_2,k_3)$, so to ease notation we will denote the Frey curves only by $E_{(a,b)}$. Whenever a discussion concerns a specific triple it will be made clear in the text.

\subsection{The conductor of $E_{(a,b)}^{(k_1,k_2,k_3)}$} Now we will prove that the curves $E = E_{(a,b)}^{(k_1,k_2,k_3)}$ satisfy the required properties to be used as Frey curves. Suppose that $(a,b,c)$ is a primitive solution to (\ref{eq1}) or (\ref{eq2}). The curves $E$ have associated the following quantities:
\begin{eqnarray*}
\Delta(E) & = & 2^{4}(ABC)^2 ,  \\ 
c_4(E) & = & 2^{4}(AB + BC + AC), \\
c_6(E) & = & -2^{5}(C + 2B)(A + 2B)(2A + B), \\
j(E) & = & 2^{8}\frac{(AB + BC + AC)^3}{(ABC)^2}. \
\end{eqnarray*}
In particular,
$$ \Delta(E) = \begin{cases}
 \mu^2 2^4(\alpha \beta \gamma)^2 c^{2p} \hspace{0.2cm} \mbox{ if } r \nmid a+b,\\
 \mu^2 2^4(\alpha \beta \gamma)^2 \pi_r^{6}c^{2p} \hspace{0.2cm} \mbox{ if } r \mid a+b. \\
\end{cases}$$

\begin{rem} We want to note that the discriminant is a constant times $p$-th power containing all the dependence on the solution. This is fundamental for the modular approach to work.
\end{rem}

Let $\mathfrak{P}$, $\pi_r$ and $\mathfrak{P}_2$ denote a prime in $K^{+}$ above $p$, $r$ and 2, respectively. Denote by $\mbox{rad}(c)$ the product of the primes dividing $c$.

\begin{pp}  Let $(a,b,c)$ be a primitive solution of equation $(\ref{eq1})$ or $(\ref{eq2})$. The conductor of the curves $E=E_{(a,b)}^{(k_1,k_2,k_3)}$ is of the form 
$$N_E = (\prod_{\mathfrak{P}_2 \mid 2} \mathfrak{P}_2^{s_i}) \pi_r^t \mbox{rad}(c),$$
where $s_i$ may be $2,3$ or $4$ and $t=0$ or $2$ if $r \mid a+b$ or $r \nmid a+b$, respectively. In particular, $E$ has good reduction at all $v \mid 3$.
\label{condEgamma}
\end{pp}
\begin{proof} To the results used in this proof we follow \cite{pap}. First note that $\alpha, \beta, \gamma$ can be written in the form $\pm \zeta^s (1 - \zeta^t) (1 - \zeta^u)$, where neither $t$ nor $u$ are $\equiv 0$ (mod $r$), which means that the only prime dividing $\alpha \beta \gamma$ is $\pi_r$ and 
$\upsilon_{\pi_r}(\alpha \beta \gamma) = 3$.\par 
Let $\mathfrak{P}$ be a prime in $K^{+}$ different from $\pi_r$ and $\mathfrak{P}_2$. Observe that $\upsilon_{\mathfrak{P}}(\Delta(E)) = 2p\upsilon_{\mathfrak{P}}(c)$. Then if $\mathfrak{P} \nmid c$ we have $\upsilon_{\mathfrak{P}}(\Delta) = 0$ and the curve has good reduction. If $\mathfrak{P} \mid c$ then $\upsilon_{\mathfrak{P}}(\Delta) > 0$ and $\mathfrak{P}$ must divide only one among $A,B$ or $C$ (see Proposition \ref{factores}). From the form of $c_4$ it can be seen that $\upsilon_{\mathfrak{P}}(c_4)=0$ thus $E$ has multiplicative reduction at $\mathfrak{P}$. \par 
From the form of $\Delta(E)$ we see that $\upsilon_{\pi_r}(\Delta) = 6$ or $12$ if $r \nmid a+b$ or $r \mid a+b$, respectively. This translate to $E$ bad additive reduction ($\upsilon_{\pi_r}(N_E) = 2$) or good reduction ($\upsilon_{\pi_r}(N_E) = 0$) at $\pi_r$ if $r \nmid a+b$ or $r \mid a+b$, respectively.\par
Since 2 do not ramifies in $\Q(\zeta)$ we use Table IV in  \cite{pap}. It is easily seen by from the shape of $\Delta$, $c_4$ and $c_6$ that $\upsilon_{\mathfrak{P}_{2}}(\Delta) = 4$, $\upsilon_{\mathfrak{P}_{2}}(c_6) = 5$ and $\upsilon_{\mathfrak{P}_{2}}(c_4) \geq 4$ for any $\mathfrak{P}_{2}$ above 2. Then the equation is minimal ($\upsilon_{\mathfrak{P}_{2}}(\Delta) < 12$) and we check in Table IV \cite{pap} for the columns corresponding to the previous valuations and observe that $\upsilon_{\mathfrak{P}_{2}}(N_{E})$ can be $2, 3, 4$ corresponding to Kodaira type II, III or IV.\par 
The last statement follows because all the primes $\q \mid c$ are above rational primes congruent to 1 modulo $r$.
\end{proof}

\begin{pp} Let $(a,b,c)$ be a primitive solution of equation $(\ref{eq1})$ or $(\ref{eq2})$ and $\ell \neq p$ be a prime in $K^{+}$ dividing $c$. Then, the representation $\bar{\rho}_{E,p}$ attached to the Frey curve $E = E_{(a,b)}$ is unramified at $\ell$.
\label{notram}
\end{pp}
\begin{proof} Note that $\ell$ is unramified in $K^{+}$ because $\ell \nmid r$. From $\ell \mid c$ and Proposition \ref{condEgamma} it follows that $\ell$ is of multiplicative reduction of $E$. Since it appears to a $p$-th power in the discriminant of a minimal model at $\ell$ of $E$ we know by a theorem from Hellegouarch that the representation $\bar{\rho}_{E,p}$ will not ramify at $\ell$.\end{proof}

\begin{cor} Suppose $\bar{\rho}_{E,p}$ to be irreducible and denote by $N(\bar{\rho}_{E,p})$ its Artin conductor outside of $p$. Then, $N(\bar{\rho}_{E,p})$ is equal to $(\prod_{\mathfrak{P}_2 \mid 2} \mathfrak{P}_2^{s_i})\pi_r^{t}$, where $s_i$ and $t$ are given as in Proposition $\ref{condEgamma}$, with the possible exception of finite small values of $p$ in which case it is a strict divisor of $(\prod_{\mathfrak{P}_2 \mid 2} \mathfrak{P}_2^{s_i})\pi_r^{t}$.
\label{artinc}
\end{cor}
\begin{proof} $N(\bar{\rho}_{E,p})$ is not divisible by primes above $p$ by definition. Also, when reducing $\rho_{E,p}$ to its residual representation $\bar{\rho}_{E,p}$, by the work of Carayol \cite{Car1} we know that the conductor at the bad additive primes will not decrease (with the possible exception of finite small values of $p$). Hence, by Propositions \ref{notram} and \ref{condEgamma} we conclude that $N(\bar{\rho}_{E,p}) = (\prod_{\mathfrak{P}_2 \mid 2} \mathfrak{P}_2^{s_i})\pi_r^{t}$.\end{proof} 

In the rest of this work we will refer to the Artin conductor outside of $p$ simply by Artin conductor. 

\subsection{Level lowering}
\label{levellowering}
The following two theorems follow from Section \ref{bgmod}. For the moment we will assume them.

\begin{tm} Let $r\geq7$ be a fixed prime and $a,b$ coprime integers. There exist a constant $M(r)$ such that for any triple $(k_1,k_2,k_3)$ (with respect to the fixed $r$) the Galois representations $\bar{\rho}_{E,\ell}$ attached to $E=E_{(a,b)}^{(k_1,k_2,k_3)}$ are absolutely irreducible for all $\ell > M(r)$.
\label{irred}
\end{tm}
\begin{tm} Suppose that $(a,b,c)$ is a primitive solution to $(\ref{eq1})$ or $(\ref{eq2})$. Then, the Frey curve $E_{(a,b)}$ over $K^{+}$ is modular.
\end{tm}

Now we will apply the level lowering results for Hilbert modular forms due to Jarvis, Rajaei and Fujiwara to show that $\bar{\rho}_{E,p}$ is modular of level $N(\bar{\rho}_{E,p}) = (\prod_{\mathfrak{P}_2 \mid 2} \mathfrak{P}_2^{s_i})\pi_r^{t}$. To simplify the following description we assume that the class number of $K^+$ is one, but the same argument holds in general (see also the discussion after Theorem 3.3 in \cite{JMee}).\par 
For an ideal $N$ of $K^{+}$ we denote by $S_2 (N)$ the set of Hilbert cuspforms of parallel weight 2, level $N$ and trivial character. Suppose that $(a,b,c)$ is a primitive solution to $(\ref{eq1})$ or $(\ref{eq2})$. It follows from modularity that there exists a newform $f_0 \in S_2 (2^{i}\pi_r^{t}\mbox{rad}(c))$ ($i\in \{2,3,4\}$ and $t \in \{0,2\}$) with coefficients field $\Q_f = \Q$, such that $\rho_{E,p}$ is isomorphic to the $p$-adic representation attached to $f_0$, which we denote by $\rho_{f_0,p}$. Then, for $p > M(r)$ we have that $\bar{\rho}_{E,p}$ is modular and irreducible and we can apply the level lowering theorems. Indeed, since $K^{+}$ might be of even degree, in order to apply the main result of \cite{Raj}, we need to add an auxiliary (special or supercuspidal) prime to 
the level. From \cite{Raj}, section 4, Theorem 5, we can add an auxiliary (special) prime $\mathfrak{q}_0$ that, in particular, satisfies that $\bar{\rho}_{f_0,p}(\mbox{Frob}_{\mathfrak{q}_0})$ is conjugated to $\bar{\rho}_{f_0,p}(\sigma)$, where $\sigma$ is complex conjugation. We now apply the main theorem of \cite{Raj} to remove from the level all primes except those above 2, $p$, the prime $\pi_r$ and $\mathfrak{q}_0$. Now we will remove from the level the primes above $p$ and for that we need $\bar{\rho}_{E,p}|G_{\mathfrak{P}}$ to be finite at all primes $\mathfrak{P} \mid p$. If $\mathfrak{P} \nmid c$ it is of good reduction for $E$ then $\bar{\rho}_{E,p}|G_{\mathfrak{P}}$ is finite; if $\mathfrak{P} \mid c$ it is of multiplicative reduction for $E$ and since we have $p \mid \upsilon_{\mathfrak{P}}(\Delta)$ it follows also that $\bar{\rho}_{E,p}|G_{\mathfrak{P}}$ is finite. Thus from Theorem 6.2 in \cite{Jarv} we can remove the primes above $p$ without changing the weight. Finally, from the 
condition imposed on $\mathfrak{q}_0$ follows that $\mbox{Nm}(\mathfrak{q}_0) \not\equiv 1$ (mod $p$) and we can apply Fujiwara version of Mazur's principle to remove $\mathfrak{q}_0$ from the level. Then we conclude that there exists a newform $f$ in $S_2(N(\bar{\rho}_{E,p}))$ and a prime $\mathfrak{P} \mid p$ in $\Q_{f}$ such that its associated residual Galois representation satisfies 
\begin{equation}
\bar{\rho}_{E,p} \sim \bar{\rho}_{f_0,p} \sim \bar{\rho}_{f,\mathfrak{P}}. 
\label{iso}
\end{equation}
Now, if we show that this congruence cannot hold for all the newforms in the corresponding cusp spaces $S_2(N(\bar{\rho}_{E,p}))$ (i.e. complete step (III)) we have proved that our putative non-trivial primitive solution $(a,b,c)$ to (\ref{eq1}) or (\ref{eq2}) cannot exist. Thus (\ref{general}) also cannot have non-trivial primitive solutions by Lemma \ref{novaeq3}.

\subsection{Step (III) and limitations of the method}
\label{limitations}

The most common method to contradict isomorphism \eqref{iso} is to look at the values $a_q(E)$ and $a_q(f)$ and verify that they cannot be congruent modulo $\mathfrak{P}$ if $p$ is greater than a constant. Unfortunately, this method is limited by the existence of trivial solutions. In particular, an intrinsic problem of what we have done so far is that for some $\mu$ the equations (\ref{eq1}) and (\ref{eq2}) have trivial solutions $\pm(1,0,1)$, $\pm(0, 1,1)$, $(1,1,1)$ and $(1,-1,1)$, $(-1,1,1)$ that correspond to the Frey curves $E_{(1,0)}$, $E_{(1,1)}$ and $E_{(1,-1)}$. Since these curves do exist we will not be able to eliminate their associated newforms simply by comparing the values of $a_q$. However, for suitable values of $C$ the extra condition $C \mid a+b$ will be enough to deal with $E_{(1,0)}$ and $E_{(1,1)}$, but the curve $E_{(1,-1)}$ will resist. To eliminate the newform corresponding to $E_{(1,-1)}$ we need the extra hypothesis $r \nmid a+b$ to achieve a contradiction at the inertia at $\pi_r$.
 From Proposition \ref{trezz2} it follows that for a primitive solution $(a,b,c)$ of (\ref{general}) we have $r \nmid a+b \Leftrightarrow r \nmid c$ thus, using the Frey curves in \eqref{Freycurve}, we are limited to prove the non-existence of primitive first case solutions to equation (\ref{general}). 

\begin{rem}
\label{exceptions}
Nevertheless, with these Frey curves, we can solve some equations completely, i.e removing the restriction $r \nmid c$. This is the case if we consider the equation $x^{2r} + y^{2r} = Cz^p$ and use the Frey curves $F_{(a,b)} := E_{(a^2,b^2)}$. Since the trivial solutions $(1,-1)$ will correspond to the curve $F_{(1,-1)} = E_{(1,1)}$, in principle, we will be able to eliminate its attached modular form because of the condition $C \mid a+b$. This will be illustrated when dealing with the case $r=7$.
\end{rem}

There are also computational limitations to the strategy, because the degree of $K^{+}$ is $(r-1)/2$ and increases with $r$. In particular, the norm of the ideals $(2)$ and $\pi_r$ will increase and consequently the dimension of the required spaces of cuspforms becomes huge very fast. For example, when $r=11$ the norm of $2^4 \pi_r^2$ is $2^{20} 11^2$ and the dimension of $S_2(2^4 \pi_r^2)$ is $5406721$. Thus, already for small values of $r$ computing the corresponding Hilbert newspace in $S_2(2^4 \pi_r^2)$ is infeasible. However, if $r \equiv 1$ (mod 6) the computational requirements can be reduced, as we will now show.

\subsection{Exponents of the form $6k+1$.} \label{descend} Let $r = 6k+1$ be a prime. The degree of $\phi_r$ is $6k$ then it admits $k$ factors $\phi_i$ for $1 \leq i \leq k$ with degree six and coefficients in the totally real subfield of $K^{+}$ of degree $k$, that we denote by $K_0$. Let $\sigma$ be a generator of $\mbox{Gal}(\Q(\zeta)/\Q)$ and let $\phi_1$ be the factor of $\phi_r$ given by 
$$ \phi_1 = \prod_{i=0}^5{(x + \sigma^{ik}(\zeta)y)}.$$  
Let also the polynomials $f_i$ be given by
$$\begin{cases}
 f_1(x,y) = (x+\zeta y)(x+\sigma^{3k}(\zeta)y),\\
 f_2(x,y) = (x+\sigma^{2k}(\zeta)y)(x+\sigma^{5k}(\zeta)y),\\
 f_3(x,y) = (x+\sigma^{k}(\zeta)y)(x+\sigma^{4k}(\zeta)y).\\
\end{cases}$$
Observe that this choice of polynomials correspond to the triple $(1,n_2,n_3)$, where $\zeta^{n_2} = \sigma^{2k}(\zeta)$ and $\zeta^{n_3} = \sigma^{4k}(\zeta)$ (we can assume $n_2 < n_3$ otherwise interchange $f_2$ and $f_3$). Note also that $\phi_1 = f_1 f_2 f_3 = f_{1} f_{n_2} f_{n_3}$ is defined over $K_0$ hence, for this triple, equations (\ref{eq1}) and (\ref{eq2}) are defined over $K_0$. As explained before this give rise to the following linear system
$$\begin{cases}
 \alpha  + \beta + \gamma = 0\\
 \alpha (\zeta +\sigma^{3k}(\zeta)) + \beta(\sigma^{2k}(\zeta) +\sigma^{5k}(\zeta)) + \gamma(\sigma^{4k}(\zeta) +\sigma^{k}(\zeta)) = 0 \\
 \alpha  + \beta + \gamma = 0\\
\end{cases}$$
and we pick the solution
$$\begin{cases}
 \alpha  = \sigma^{4k}(\zeta) + \sigma^{k}(\zeta) - \sigma^{2k}(\zeta) - \sigma^{5k}(\zeta)\\
 \beta   = \zeta + \sigma^{3k}(\zeta)             - \sigma^{4k}(\zeta) - \sigma^{k}(\zeta)\\
 \gamma  = \sigma^{2k}(\zeta) + \sigma^{5k}(\zeta) - \zeta - \sigma^{3k}(\zeta) \\
\end{cases}$$
Observe that this choice for $(\alpha, \beta, \gamma)$ could also be obtained by replacing in the general solution for $(\alpha,\beta,\gamma)$ the triple $(k_1,k_2,k_3)$ by $(1,n_2,n_3)$. 
Again, we let $A_{(a,b)} = \alpha f_1(a,b)$ , $B_{(a,b)} = \beta f_2(a,b)$ , $C_{(a,b)} = \gamma f_3(a,b)$, hence we have $A + B + C = 0$ and, as before, we can consider the Frey curve over $K^{+}$ given by
$$E_{(a,b)}:= E_{(a,b)}^{(1,n_2,n_3)} : Y^2 = X(X-A_{(a,b)})(X+B_{(a,b)}).$$

For the rest of this section we have fixed  $(k_1,k_2,k_3)=(1,n_2,n_3)$. 
\begin{pp} Let $r=6k+1 \geq 7$ be a prime and $(k_1, k_2, k_3)= (1,n_2,n_3)$. Suppose that $(a,b,c)$ is a primitive solution of $(\ref{eq1})$ or $(\ref{eq2})$. Then the Frey curves $E_{(a,b)}/K^{+}$ has a model over $K_0$.
\label{isos}
\end{pp}
\begin{proof} First observe that $\sigma^{2k}$ (mod $\sigma^{3k}$) has order 3 and generates $\mbox{Gal}(K^{+}/K_0)$. Furthermore, 
$\sigma^{2k}(\alpha)=\beta$, $\sigma^{2k}(\beta)=\gamma$, $\sigma^{2k}(\gamma)=\alpha$ and also $\sigma^{2k}(f_1)=f_2$, $\sigma^{2k}(f_2)=f_3$, 
$\sigma^{2k}(f_3)=f_1$, then
$$ \sigma^{2k}(A) = B , \hspace{1cm} \sigma^{2k}(B) = C , \hspace{1cm}  \sigma^{2k}(C) = A.$$  
Now we write $E_{(a,b)}$ in the short Weierstrass form to get a model
$$\begin{cases}
   E_{(a,b)} : Y^2 = X^3 + a_4 X + a_6, \mbox{ where } \\
   a_4 = -432(AB + BC + CA) \\
   a_6 = -1728 (2A^3 + 3A^2B - 3AB^2 - 2B^3) \\
\end{cases}$$
Since $a_4$ is clearly invariant under $\sigma^{2k}$ and
\begin{eqnarray*}
  a_6  & = & -1728 (2A^3 + 3A^2B - 3AB^2 - 2B^3) = \\
       & = &  -1728 (2(-B-C)^3 + 3(-B-C)^2B - 3(-B-C)B^2 - 2B^3) = \\
       & = & -1728 (2B^3 + 3B^2C - 3BC^2 - 2C^3) = \sigma^{2k}(a_6)  \
\end{eqnarray*}
we conclude that the short Weierstrass model is defined over $K_0$.\end{proof}

Let $\pi_2$ and $\pi_r$ denote a prime in $K_0$ above 2 and $r$, respectively.
\begin{pp} Let $r=6k+1 \geq 7$ be a prime. Fix the triple $(k_1, k_2, k_3)$ to be $(1,n_2,n_3)$. Suppose that $(a,b,c)$ is a primitive solution of $(\ref{eq1})$ or $(\ref{eq2})$. The conductor of the curves $E=E_{(a,b)}$ over $K_0$ is of the form 
$$N_E = (\prod_{\mathfrak{P}_2 \mid 2} \mathfrak{P}_2^{s_i}) \pi_r^2 \mbox{rad}(c),$$
where $s_i$ may be $2,3$ or $4$.
\label{cond3}
\end{pp}
\begin{proof} Writing a curve in short Weierstrass form changes the values of $\Delta$, $c_4$ and $c_6$ by a factor of $6^{12}$, $6^4$ and $6^6$. Since the primes dividing $6$ do not ramify in $K/K_0$ and do not divide $c$ the conductor of $E$ at primes dividing $c$ is the same as before.\par 
Since $\pi_r$ factorizes as the cube of the ideal above $r$ in $K^{+}$ we see from the third paragraph in the proof of Proposition \ref{condEgamma} that $\upsilon_{\pi_r}(\Delta(E)) = 4$ or $2$. Also, $\upsilon_{\pi_r}(c_4(E)) > 0$ and since we are in characteristic $\geq 5$ this implies that the equation is minimal and has bad additive reduction with $\upsilon_{\pi_r} (N_{E}) = 2$.\par
It easily can be seen that $\upsilon_{\pi_2}(\Delta(E)) = 16$, $\upsilon_{\pi_2}(c_6(E)) = 11$ and $\upsilon_{\pi_2}(c_4(E)) \geq 8$. Table IV in \cite{pap} tell us that the equation is not minimal and after a change of variables we have $\upsilon_{\pi_2}(\Delta(E)) = 4$, $\upsilon_{\pi_2}(c_6(E)) = 5$ and $\upsilon_{\pi_2}(c_4(E)) \geq 4$. Now exactly as in the proof of Proposition \ref{condEgamma} we can conclude that $\upsilon_{\pi_2}(N_{E})$ may be  2, 3, or 4.\end{proof}

The existence of a model over $K_0$ has advantages: a direct adaptation of the proof of Theorem \ref{irr} below will give us a smaller constant $M(r)$; also, by arguing exactly as in the proofs of Proposition \ref{notram} and Corollary \ref{artinc}, where instead of Proposition \ref{condEgamma} we use Proposition \ref{cond3}, it follows

\begin{pp} Let $(a,b,c)$ be a primitive solution of equation $(\ref{eq1})$ or $(\ref{eq2})$. Then, 
\begin{enumerate}
 \item If $\ell \neq p$ is a prime dividing $c$ then the representation $\bar{\rho}_{E,p}$ attached to the Frey curve $E = E_{(a,b)}$ is unramified at $\ell$.
 \item The Artin conductor $N(\bar{\rho}_{E,p})$ is equal to $(\prod_{\mathfrak{P}_2 \mid 2} \mathfrak{P}_2^{s_i})\pi_{r}^2$, where $s_i$ is as in Proposition $\ref{cond3}$.
\end{enumerate}
\label{artinc2I}
\end{pp}
Moreover, assuming modularity of the curves $E_{(a,b)} / K_0$ we can argue exactly as we did over $K^{+}$ to apply the results on level lowering. This leads to the computation of Hilbert newforms over $K_0$ which is a number field of dimension $k$ when \textit{a priori} we were over $K^{+}$ of dimension $3k$. In section \ref{sete} we will take advantage of this to solve equations for $r=7$.

\section{Constructing Frey curves part II}
\label{curvasII}

Let $r \geq 7$ be a prime and fix $\zeta := \zeta_r$. Let $K^{+}$ the maximal totally real subfield of $\Q(\zeta)$. Now, instead of choosing three factors $f_{k_i}$ we will use only two, together with the polynomial $(x+y)^2$. That is, we pick 
$$ f_{k_1} = x^2 + (\zeta^{k_1} + \zeta^{-k_1}) xy + y^2 \quad \quad  f_{k_2} = x^2 + (\zeta^{k_2} + \zeta^{-k_2}) xy + y^2 $$
and we want to find $(\alpha,\beta,\gamma)$ such that
$$\alpha (x+y)^2 + \beta f_{k_1} + \gamma f_{k_2} = 0. $$
In particular, we can take
$$\begin{cases}
 \alpha = \zeta^{k_2} + \zeta^{-k_2} - \zeta^{k_1} - \zeta^{k_1}, \\
 \beta = 2 - \zeta^{k_2} - \zeta^{-k_2}, \\
 \gamma = \zeta^{k_1} + \zeta^{-k_1} - 2. \
\end{cases}$$
Write $A_{(a,b)} = \alpha (a+b)^2$, $B_{(a,b)} = \beta f_{k_1}(a,b)$, $C_{(a,b)} = \gamma f_{k_2}(a,b)$ and define
\begin{equation}
E_{(a,b)}^{(k_1,k_2)} : Y^2 = X(X-A_{(a,b)})(X + B_{(a,b)})
\label{FreycurveII} 
\end{equation}
 
\subsection{The conductor of $E_{(a,b)}^{(k_1,k_2)}$}

Let $C$ be an integer divisible only by primes $q \not\equiv 1,0 \pmod{r}$. Suppose we have a putative primitive solution $(a,b,c)$ to the equation $x^r + y^r = Cz^p$.\\

Recall that $\pi_r$ is the only prime in $K^+$ dividing $r$. From Lemma \ref{novaeq3} there is $c_0 \in \Z$ and $c_1,c_2 \in \OO_{K^+}$ all dividing $c$ such that 
\begin{enumerate}
\item $a+b=Cc_0^p$
\item $f_{k_i}(a,b) = \mu_i c_i^p$ for some $\mu_i \in \OO_{K^+}^{\times}$ for $i=1,2$
\end{enumerate}
if $r \nmid a+b$ or,
\begin{enumerate}
\item $a+b=Cr^{k}c_0^p$ with $k>0$.
\item $f_{k_i}(a,b) = \mu_i \pi_r c_i^p$ for some $\mu_i \in \OO_{K^+}^{\times}$ for $i=1,2$
\end{enumerate}
if $r \mid a+b$. Moreover, $c_0,c_1,c_2$ are pairwise coprime and $c_1,c_2$ are divisible only by primes above rational primes congruent to 1 modulo $r$. The Frey curve $E_{(a,b)}^{(k_1,k_2)}$ has the invariants 
$$\begin{cases}
   \Delta(E) = & 2^{4}(\alpha\beta\gamma)^2 (a+b)^4(c_1c_2)^{2p} ,  \\ 
    c_4(E) = & 2^{4}(\alpha\beta (a+b)^2 c_1^p + \beta\gamma (c_1 c_2)^p + \alpha\gamma (a+b)^2 c_2^p), \\
    c_6(E) = & -2^{5}(\gamma c_2^p + 2\beta c_1^p)(\alpha (a+b)^2 + 2\beta c_1^p)(2\alpha (a+b)^2 + \beta c_1^p), \
\end{cases}$$    
if $r \nmid a+b$. If $r \mid a+b$ the invariants are given by
$$\begin{cases}
   \Delta(E) = & 2^{4}(\alpha\beta\gamma)^2 \pi_r^2 (a+b)^4(c_1c_2)^{2p} ,  \\ 
    c_4(E) = & 2^{4}(\alpha\beta \pi_r (a+b)^2 c_1^p + \beta\gamma \pi_r^2 (c_1 c_2)^p + \alpha\gamma \pi_r (a+b)^2 c_2^p), \\
    c_6(E) = & -2^{5}(\gamma \pi_r c_2^p + 2\beta \pi_r c_1^p)(\alpha (a+b)^2 + 2\beta \pi_r c_1^p)(2\alpha (a+b)^2 + \beta \pi_r c_1^p). \
\end{cases}$$    

In the previous formulas we did not included the units $\mu_i$ since they make no difference in what follows. For an element $c \in \OO_{K^+}$ let $\mbox{Rad}_{2r}(c)$ denote the product of the primes ideals $\p \nmid 2r$ that divide $c$.

\begin{pp} Fix $r\geq 7$ and $(k_1,k_2)$ as above. The conductor $N_E$ of the curves $E = E^{(k_1,k_2)}_{(a,b)}$ is of the form 
$$ (\prod_{\mathfrak{P}_2 \mid 2} \mathfrak{P}_2^{s_i}) \pi_r^t \mbox{Rad}_{2r}((a+b)c_1c_2) \quad \mbox{ where } \quad s_i \in \{0,1,2,3,4 \},$$
and $t=1$ or $2$ if $r \mid a+b$ or $r \nmid a+b$, respectively. In particular, the curve $E$ is semistable at all primes $v \mid 3$
\label{condEgammaII}
\end{pp}
\begin{proof} First note that $\alpha, \beta, \gamma$ can be written in the form $\pm \zeta^s (1 - \zeta^t) (1 - \zeta^u)$, where neither $t$ nor $u$ are $\equiv 0$ (mod $r$), which means that the only prime dividing $\alpha \beta \gamma$ is $\pi_r$ and $\upsilon_{\pi_r}(\alpha \beta \gamma) = 3$.\par 
Let $\mathfrak{P} \nmid 2r$ be a prime in $K^{+}$. Observe that $\upsilon_{\mathfrak{P}}(\Delta(E)) = 4\upsilon_{\mathfrak{P}}(a+b)$ + $2p\upsilon_{\mathfrak{P}}(c_1c_2)$. Then if $\mathfrak{P} \nmid (a+b)c_1c_2$ we have $\upsilon_{\mathfrak{P}}(\Delta) = 0$ and the curve has good reduction; if $\mathfrak{P} \mid (a+b)c_1c_2$ then $\upsilon_{\mathfrak{P}}(\Delta) > 0$ and $\upsilon_{\mathfrak{P}}(c_4)=0$, i.e multiplicative reduction at $\mathfrak{P}$ (see Table I in \cite{pap}).\par 
If $\pi_r \nmid a+b$: $\upsilon_{\pi_r}(c_4) \geq 2$, $\upsilon_{\pi_r}(c_6) \geq 1$ and $\upsilon_{\pi_r}(\Delta) = 6$. Thus $E$ has bad additive reduction and $\upsilon_{\pi_r}(N_E) = 2$; if $\pi_r \mid a+b$: $\upsilon_{\pi_r}(c_4) =  4$, $\upsilon_{\pi_r}(c_6) \geq 6$ and $\upsilon_{\pi_r}(\Delta) = 8 + 4\upsilon_{\pi_r}(a+b)$. Since $\upsilon_{\pi_r}(a+b) \geq (r-1)/2$ the equation is non-minimal and after a coordinate change we have $\upsilon_{\pi_r}(c_4) = 0$ and $\upsilon_{\pi_r}(\Delta) > 0$. Thus  $E$ has bad multiplicative reduction, i.e. $\upsilon_{\pi_r}(N_E) = 1$.\par
Let $\mathfrak{P}_2 \mid 2$ be a prime. Since 2 do not ramify in $\Q(\zeta)$ we use Table IV in  \cite{pap}. If $2 \mid a+b$: $\upsilon_{\mathfrak{P}_{2}}(c_4) = 4$, $\upsilon_{\mathfrak{P}_{2}}(c_6) = 6$ and $\upsilon_{\mathfrak{P}_{2}}(\Delta) = 4 + 4\upsilon_{\mathfrak{P}_{2}}(a+b)$. The corresponding table entries give us $2,3$ or $4$ as possible valuations at $\mathfrak{P}_{2}$ of the conductor or the equation for $E$ is non-minimal. In the non-minimal case, after change of variables, we obtain semistable reduction, i.e $\upsilon_{\mathfrak{P}_{2}}(N_E)=0$ or 1; if $2 \nmid a+b$: $\upsilon_{\mathfrak{P}_{2}}(c_4) \geq 4$, $\upsilon_{\mathfrak{P}_{2}}(c_6) = 5$ and $\upsilon_{\mathfrak{P}_{2}}(\Delta) = 4$ for any $\mathfrak{P}_{2}$ above 2. The possible valuations of the conductor at $\mathfrak{P}_{2}$ for these entries are $2,3,4$.\par
The last statement follows immediately because the additive primes divide $2r$.
\end{proof}
Now let $\q \nmid 2r$ be a multiplicative prime of $E$, that is $\q \mid (a+b)(c_1c_2)$. Since $a+b=Cr^kc_0^p$ then $\q$ appears to a $p$-th power in the minimal discriminant at $\q$ of $E$ if and only if $\q \nmid C$. With this in mind, in the same lines of the proofs of Proposition \ref{notram} and Corollary \ref{artinc} we can prove

\begin{pp} Let $(a,b,c)$ be a primitive solution of the equation $x^r + y^r = Cz^p$. Write $E=E_{(a,b)}^{(k_1,k_2)}$.   
\begin{enumerate}
 \item Let $\q \nmid 2rp$ be a prime in $K^{+}$ dividing $c$. Then, the representation $\bar{\rho}_{E,p}$ is unramified at $\q$.
 \item Suppose $\bar{\rho}_{E,p}$ to be irreducible. Then, the Artin conductor (outside $p$) $N(\bar{\rho}_{E,p})$ is of the form 
$$(\prod_{\mathfrak{P}_2 \mid 2} \mathfrak{P}_2^{s_i}) \pi_r^t \mbox{Rad}_{2r}(C)  \quad \mbox{ where } \quad s_i \in \{0,1,2,3,4 \},$$
and $t=1$ or $t=2$ if $r \mid a+b$ or $r \nmid a+b$, respectively. 
\end{enumerate}
\label{artinc2}
\end{pp}

\subsection{Level lowering and Step (III)}
\label{ll2}

As in the case of the Frey curves in \eqref{Freycurve} the the following theorems also follow from Section \ref{bgmod}. For the moment we will assume them.

\begin{tm} Let $r\geq7$ be a fixed prime and $a,b$ coprime integers. There exist a constant $M(r)$ such that for any pair $(k_1,k_2)$ the Galois representations $\bar{\rho}_{E,\ell}$ attached to $E=E_{(a,b)}^{(k_1,k_2)}$ are absolutely irreducible for all $\ell > M(r)$.
\label{irred22}
\end{tm}
\begin{tm} Fix $r \geq 7$ and $(k_1,k_2)$. Suppose that $(a,b,c)$ is a primitive solution to $x^r + y^r = Cz^p$. Then, the Frey curve $E_{(a,b)}^{(k_1,k_2)}$ over $K^{+}$ is modular.
\end{tm}

Suppose that $(a,b,c)$ is a primitive solution to $x^r + y^r = Cz^p$ and write $E=E_{(a,b)}^{(k_1,k_2)}$. We can argue exactly as in section \ref{levellowering} and apply level lowering to conclude that there is some Hilbert newform $f \in S_2(N(\bar{\rho}_{E,p}))$ and a prime $\mathfrak{P} \mid p$ such that the following isomorphism holds
\begin{equation}
\bar{\rho}_{E,p} \sim \bar{\rho}_{f,\mathfrak{P}}. 
\label{isoII}
\end{equation}  

We now need to contradict \eqref{isoII} to complete step (III). As discussed in section \ref{limitations} there are obstacles to this, in particular, arising from trivial solutions. Observe now that if $|C| \geq 3$ then $(1,0,1)$ and $(1,1,1)$ are not solutions of $x^r + y^r = Cz^p$. There is still the trivial solution $(1,-1,0)$ but we see from the formula for the discriminant of $E$ that $E_{(1,-1)}$ is singular, hence it is not an elliptic curve. Then, the spaces $S_2(N(\bar{\rho}_{E,p}))$ will not contain a newform associated with it. This removes the constraint $r \nmid c$ that we needed to introduce in section \ref{limitations}. Indeed, we can now easily prove the following theorem

\begin{tm} Let $C \neq 1,2$ be an integer divisible only by primes $q \not\equiv 1,0 \pmod{r}$. Let $N(\bar{\rho}_{E,p})$ be given as in Proposition $\ref{artinc2}$. Suppose that the spaces $S_2(N(\bar{\rho}_{E,p}))$ contain no newforms corresponding to elliptic curves with full $2$-torsion. Then, there are no primitive solutions to the equation $x^r + y^r = Cz^p$ for $p$ greater than a constant.
\end{tm}
\begin{proof} Since $C \neq 1,2$, the spaces $S_2(N(\bar{\rho}_{E,p}))$ do not contain newforms associated with Frey curves attached to the trivial solutions. Moreover, the prime factors of $C$ divide $N(\bar{\rho}_{E,p})$ only once. Hence, all the rational newform in these spaces will correspond to elliptic curves by Eichler-Shimura.

Suppose that there are no rational newforms in $S_2(N(\bar{\rho}_{E,p}))$ corresponding to elliptic curves with full $2$-torsion. Then, we have an isomorphism $\bar{\rho}_{E,p} \sim \bar{\rho}_{f,\mathfrak{P}}$ where
$f$ is a newform with field of coefficients strictly containing $\Q$ or $f$ corresponds to an elliptic curve without full $2$-torsion. In both cases it is known that for each $f$ the isomorphism cannot hold for $p$ greater than a constant $M_f$. Since there are only a finite list of $f$ the result follows.
\end{proof}

Unfortunately, the Frey curves \eqref{FreycurveII} have computational disadvantages. For example, for $r=7$ we work over $\Q$ if we use the curves \eqref{Freycurve} and over the cubic field $K^+ \subset \Q(\zeta_7)$ if we use the curves \eqref{FreycurveII}. Nevertheless, for primes of the form $r=4m+1$ we are able to diminish the required computations as we now explain.\par 
Suppose $r=4m+1$ is a prime and note that there is a subfield $k \subset K^+$ such that $[K^+ : k]=2$ and $[k:\Q]=m$. Let $\sigma$ be a generator of $\mbox{Gal}(K^+/\Q)$, hence $\sigma^m$ generates $\mbox{Gal}(K^+/k)$. Pick a pair $(k_1, n_2)$ such that 
$$\sigma^{m}(\zeta^{k_1} + \zeta^{-k_1}) = \zeta^{n_2} + \zeta^{-n_2}$$ 
and consider the corresponding $A_{(a,b)},B_{(a,b)},C_{(a,b)}$ as above. It is easy to see that
$$\sigma^m (A_{(a,b)}) = - A_{(a,b)}, \quad \sigma^m (B_{(a,b)}) = - C_{(a,b)} \quad \sigma^m (C_{(a,b)}) = - B_{(a,b)},$$
hence the formulas in the proof of Proposition \ref{isos} show that the curve $E_{(a,b)}^{(k_1,n_2)}$ has a short Weierstrass model defined over $k$.\\ 

The first prime where we can profit from this is $r=13$, but we would need to compute Hilbert newforms over the cubic field inside $\Q(\zeta_{13})$ for levels that makes the computation impossible. Nevertheless, in \cite{DF2}, taking advantage of the fact $13=6\times+1$ the authors work with the Frey curves \eqref{FreycurveII} over $\Q(\sqrt{13})$ and are able to prove the non-existence of non-trivial primitive first case solutions to $x^{13} + y^{13} =Cz^p$, for infinitely many values of $C$. This is another example of the computational disadvantages of using \eqref{FreycurveII}.

\section{Constructing Frey curves part III}
\label{4m}

In this section we will construct more Frey curves attached to solutions of the equation $x^r + y^r = Cz^p$ for exponents of the form $r=4m + 1$. These curves will have the property of being $k$-curves. This will be achieved in two steps. First, using observations about the factors of $\phi_r$ as in section \ref{relating}, we will again relate different equations. Second, we will generalize the  ideas that led to the construction of the Frey $\Q$-curves in \cite{DF}.\\

Let $r=4m + 1$ be a prime and $\zeta := \zeta_r$ a $r$-th primitive root of unity. Recall that $K^{+}$ is the maximal totally real subfield of the cyclotomic field $\Q(\zeta)$. Then $K^{+}$ has degree $2m$ and there exists a subfield $k \subset K^{+}$ such that $[K^{+}:k]=2$ and $[k:\Q]=m$. Let $\sigma$ be the generator of $\mbox{Gal}(K^{+}/\Q)$ then $\sigma^m$ generates $\mbox{Gal}(K^{+}/k)$. Recall that $x^r + y^r = (x+y)\phi_r(x,y)$ and $\phi_r(x,y)$ factors as a product of $2m$ degree two polynomials with coefficients in $K^{+}$. Pick a pair $(k_1, n_2)$ such that 
$$\sigma^{m}(\zeta^{k_1} + \zeta^{-k_1}) = \zeta^{n_2} + \zeta^{-n_2}$$  
and consider the factors of $\phi_r$ given by
$$ f_{k_1} = x^2 + (\zeta^{k_1} + \zeta^{-k_1}) xy + y^2 \quad \quad  f_{n_2} = x^2 + (\zeta^{n_2} + \zeta^{-n_2}) xy + y^2.$$

Let $h_r^{+}$ be the class number of $K^{+}$ and $\pi_r$ be such that $r\mathcal{O}_{K^{+}} = (\pi_r)^{(r-1)/2}$. Let $p$ be a prime not dividing $h_r^{+}$. Suppose there is a primitive solution $(a,b,c')$ to $x^r+y^r=C z^p$ with $C\not=0$ an integer divisible only by primes $q \neq r$ satisfying $q \not\equiv 1\pmod{r}$. From Lemma \ref{novaeq3} there are $c_1,c_2 \in \OO_{K^+}$ dividing $c'$ such that 
\begin{enumerate}
\item $f_{k_1}(a,b) = \mu_1 c_1^p$ for some $\mu_1 \in \OO_{K^+}^{\times}$,
\item $f_{n_2}(a,b) = \mu_i c_2^p$ for some $\mu_2 \in \OO_{K^+}^{\times}$
\end{enumerate}
if $r \nmid a+b$ or,
\begin{enumerate}
\item $f_{k_1}(a,b) = \mu_1 \pi_r c_1^p$ for some $\mu_1 \in \OO_{K^+}^{\times}$,
\item $f_{n_2}(a,b) = \mu_i \pi_r c_2^p$ for some $\mu_2 \in \OO_{K^+}^{\times}$
\end{enumerate}
if $r \mid a+b$. Moreover, $c_1,c_2$ are coprime and divisible only by primes above rational primes congruent to 1 modulo $r$. 
Hence, by multiplication we have can transform the solution $(a,b,c')$ in a solution $(a,b,c) \in \mathbb{Z}^2 \times \mathcal{O}_{K^+}$ of the equation
\begin{equation}
\label{eq111}
f_{k_1}(x,y) f_{n_2}(x,y) = \mu z^p \quad \mbox{ or}
\end{equation} 
\begin{equation}
\label{eq222}
f_{k_1}(x,y) f_{k_2}(x,y) = \mu \pi_r^{2} z^p,
\end{equation}
which satisfies $r\nmid a+b$ in case $(\ref{eq111})$ and $r\mid a+b$ in case $(\ref{eq222})$. We are now going to construct Frey curves attached to solutions of equations (\ref{eq111}) or (\ref{eq222}). The curves will be defined over $K^{+}$ and in section \ref{kcurvas} we will show that they are $k$-curves. In order to construct a useful $k$-curve we first need to find $\alpha$, $\beta$ such that
$$(a+b)^2 = \alpha f_{k_1}(a,b) + \beta f_{n_2} (a,b).$$ 
That is, solve the linear system
$$\begin{cases}
 \alpha + \beta = 1 \\
 \alpha (\zeta^{k_1} + \zeta^{-k_1}) + \beta (\zeta^{n_2} + \zeta^{-n_2}) = 2, \\
\end{cases}$$
which has a solution for  $\alpha, \beta \in K^{+}$ given by
$$\begin{cases}
   \alpha = (\zeta^{n_2} + \zeta^{-n_2} - 2)(\zeta^{n_2} + \zeta^{-n_2} - \zeta^{-k_1} - \zeta^{k_1})^{-1} \\
   \beta = (2 - \zeta^{k_1} - \zeta^{-k_1})(\zeta^{n_2} + \zeta^{-n_2} - \zeta^{k_1} - \zeta^{-k_1})^{-1},\\
\end{cases}$$
that easily can be seen to satisfy $\sigma^m(\alpha) = \beta$.\\

Finally, we can define the Frey curves. Suppose there is a putative solution $(a,b,c) \in \mathbb{Z}^2 \times \mathcal{O}_{K^+}$ to (\ref{eq111}) or (\ref{eq222}). We consider Frey curves over $K^+$ given by
\begin{equation}
E_{(a,b)}^{(k_1,n_2,+)} : Y^2 = X^3 + 2(a+b)X^2 + \alpha f_{k_1} (a,b) X.
\label{FreycurveIII} 
\end{equation}

\begin{rem}
If instead we start by looking for $\alpha, \beta$ such that 
$$(a-b)^2 = \alpha f_{k_1}(x,y) + \beta f_{n_2}(x,y)$$
we can apply a similar construction to obtain the Frey curves 
\begin{equation}
E_{(a,b)}^{(k_1,n_2,-)} : Y^2 = X^3 + 2(a-b)X^2 + \alpha f_{k_1} (a,b) X.
\label{FreycurveIIII} 
\end{equation}
Similar properties to those we prove in the next sections for $E_{(a,b)}^{(k_1,n_2,+)}$ also hold for $E_{(a,b)}^{(k_1,n_2,-)}$.
\end{rem}

\subsection{The conductor of $E_{(a,b)}^{(k_1,n_2,+)}$} In this section we show that the elliptic curves $E = E_{(a,b)}^{(k_1,n_2,+)}$ satisfy the properties required to be used as Frey curves. Note that $E$ is of the form $Y^2 = X^3 + a_2 X^2 + a_4X$, hence 
\begin{equation*}
 \Delta(E) = 16a_4^2(a_2^2 - 4a_4) = 16\alpha^2 f_{k_1}^2 (4(a+b)^2 - 4\alpha f_{k_1}) = 64\alpha^2 \beta f_{k_1}^2 f_{n_2}. 
\end{equation*}
In particular, for a putative solution $(a,b,c)$ of equation (\ref{eq111}) or (\ref{eq222}) we have $c=c_1c_2$ and
$$ \Delta(E) = \begin{cases}
 \mu \mu_1 2^6 \alpha^2 \beta (c_1^2c_2)^{p} \hspace{0.2cm} \mbox{ if } r \nmid a+b,\\
 \mu \mu_1 2^6 \alpha^2 \beta \pi_r^{3}(c_1^2c_2)^{p} \hspace{0.2cm} \mbox{ if } r \mid a+b. \\
\end{cases}$$
Moreover, the following quantities are also associated with $E$
\begin{eqnarray*}
c_4(E) & = & 2^{4}(\alpha f_{k_1} + 2^2 \beta f_{n_2}), \\
c_6(E) & = & 2^{6}(a+b)(\alpha f_{k_1} - 2^3 \beta f_{n_2}) \\
\end{eqnarray*}
\begin{pp} Let $(a,b,c)$ be a primitive solution to $(\ref{eq111})$ or $(\ref{eq222})$. Then, the conductor of the curve $E$ is of the form
$$ N_E = \begin{cases}
 (\prod_{\mathfrak{P}_2 \mid 2} \mathfrak{P}_2^{s_i}) \mbox{rad}(c) \hspace{0.2cm} \mbox{ if } r \nmid a+b,\\
 (\prod_{\mathfrak{P}_2 \mid 2} \mathfrak{P}_2^{s_i}) \pi_r^2 \mbox{rad}(c) \hspace{0.2cm} \mbox{ if } r \mid a+b, \\
\end{cases}$$
where $s_i=5$ or $6$. In particular, the curve $E$ has good reduction at all $v \mid 3$.
\label{cond222}
\end{pp}
\begin{proof} Recall $c = c_1c_2$ and $\mu, \mu_i$ are units. We will now see that $\alpha, \beta$ are also units. Note that the elements in $K^+$ of the form $ \zeta^c(1-\zeta^a)(1-\zeta^b)$, where $a,b \not\equiv 0 \pmod{r}$ are divisible only by the prime ideal $(\pi_r) \mid r$ and their $(\pi_r)$-adic valuation is 1. Moreover, 
$$\zeta^c(1-\zeta^a)(1-\zeta^b) = \zeta^c - \zeta^{a+c} - \zeta^{b+c} + \zeta^{a+b+c} = \zeta^x - \zeta^y - \zeta^z + \zeta^{y+z-x},$$
where the last equality is a change of variables. In particular, $\zeta^{n_2} - \zeta^{k_1} - \zeta^{-k_1} + \zeta^{-n_2}$, the denominator of $\beta$, is of this form. Also, $2 - \zeta^{k_1} - \zeta^{-k_1} = (1 - \zeta^{k_1})(1- \zeta^{-k_1})$ is the numerator of $\beta$. Since both numerator and denominator of $\beta$ are of the shape above we conclude that $\beta$ is a unit, hence $\alpha = \sigma^m(\beta)$ is also a unit.\par 
Let $\mathfrak{q} \mid c$ be a prime in $K^{+}$. We have $\upsilon_{\mathfrak{q}}(\Delta) > 0$ and since $\mathfrak{q}$ divides only one of the $c_i$ it is clear from the form of $c_4$ that $\upsilon_{\mathfrak{q}}(c_4)=0$, thus the reduction is multiplicative at $\mathfrak{q}$. If $\q \nmid 2\pi_r c$ then $\upsilon_{\mathfrak{q}}(\Delta) = 0$ and $\q$ is of good reduction.\par
Let $\mathfrak{q} \mid 2$. We have $\upsilon_{\mathfrak{q}}(\Delta) = 6$, $\upsilon_{\mathfrak{q}}(c_4) = 4$, $\upsilon_{\mathfrak{q}}(c_6) = 6$, hence the equation is minimal and $\upsilon_{\mathfrak{q}}(N_E) = 5$ or $6$ by Table IV in \cite{pap}.\par
Let $\mathfrak{q} = (\pi_r)$. Then $\upsilon_{\mathfrak{q}}(\Delta) = 0$ or $\upsilon_{\mathfrak{q}}(\Delta) = 3$ if $r \nmid a+b$ or $r \mid a+b$, respectively. In particular, $\upsilon_{\mathfrak{q}}(N_E) = 0$ (good reduction) if $r \nmid a+b$. Suppose $r \mid a+b$, then $\upsilon_{\mathfrak{q}}(f_{k_1}(a,b))= \upsilon_{\mathfrak{q}}(f_{n_2}(a,b)) = 1$, hence $\upsilon_{\mathfrak{q}}(c_4) > 0$. Thus $E$ has additive reduction and $\upsilon_{\mathfrak{q}}(N_E) = 2$ by Table I in \cite{pap}.\par
The last statement follows because all the primes $\q \mid c$ are above rational primes congruent to 1 modulo $r$.\end{proof}

The next theorem shows that the curves $E_{(a,b)}^{(k_1,n_2,+)}$ have the necessary properties in order to be used as Frey curves. 

\begin{tm} Let $(a,b,c)$ be a primitive solution to $(\ref{eq111})$ or $(\ref{eq222})$ and write $E:=E_{(a,b)}^{(k_1,n_2,+)}$. Then,
\begin{enumerate}
\item The Artin conductor of $\bar{\rho}_{E,p}$ is not divisible by primes dividing $pc$.
\item $\bar{\rho}_{E,p}$ is finite at all primes $\mathfrak{p}$ dividing $p$.
\end{enumerate}
\end{tm}
\begin{proof} Part (1) follows from Proposition \ref{cond222} and the same argument as in the proof of Proposition \ref{notram}. Let $\mathfrak{p} \mid p$. If $\mathfrak{p} \nmid c$ it is of good reduction for $E$ then $\bar{\rho}_{E,p}$ is finite; if $\mathfrak{p} \mid c$ it is of multiplicative reduction for $E$ and since we have $p \mid \upsilon_{\mathfrak{p}}(\Delta)$ it follows that $\bar{\rho}_{E,p}$ is also finite. This proves (2).\end{proof}

Modularity of $E=E_{(a,b)}^{(k_1,n_2,\pm)}$ and irreducibility of $\bar{\rho}_{E,p}$ for $p$ greater than a constant also follow from the theorems in Section \ref{bgmod}. Thus, we can again apply level lowering to conclude that we have an isomorphism $\bar{\rho}_{E,p} \sim \bar{\rho}_{f,\mathfrak{P}}$ where $f \in S_2(N(\bar{\rho}_{E,p}))$. To complete step (III) we want to contradict this isomorphism. In these spaces there will be newforms corresponding to the trivial solutions. Nevertheless, we may be able to contradict the previous isomorphism for these `bad' newforms because of the condition $C \mid a+b$ and the fact that the trivial solution $(1,-1,0)$ corresponds to an elliptic curve with complex multiplication. This situation, for $r=5$, is treated in detail in \cite{DF}.

\subsection{The Frey curves as $k$-curves.}
\label{kcurvas}

In this section we will show that the Frey curves $E_{(a,b)}^{(k_1,n_2,+)}$ are indeed $k$-curves. We start by recalling the definition of $k$-curve.

\begin{defin} Let $k$ be a number field and $G_k = \mbox{Gal}(\bar{\Q}/k)$ its absolute Galois group. We will say that an elliptic curve $C$ over $\bar{k}$ is a $k$-curve if for every $\sigma \in G_k$ there exists an isogeny
$\phi_\sigma : {^\sigma}C \rightarrow C$ defined over $\bar{k}$. We say that a $k$-curve $C$ is \textit{completely defined} over a number field $K \supset k$ if all the conjugates of $C$ and the isogenies between them are defined over $K$.
\end{defin}

Recall that $\sigma^m$ generates $\mbox{Gal}(K^{+} / k)$. The Galois conjugate of $E:=E_{(a,b)}^{(k_1,n_2,+)}$ by $\sigma^m$ is given by
$${^{\sigma^m}}E : Y^2 = X^3 + 2(a+b)X^2 + \beta f_2 (a,b) X ,$$
and they admit a 2-isogeny $\mu : {^{\sigma^m}}E \rightarrow E$ given by
$$(X,Y) \mapsto (-\frac{Y^2}{2X^2},\frac{\sqrt{-2}}{4}\frac{Y}{X^2}(-\beta f_2 + X^2)).$$
The dual isogeny $\hat{\mu} : E \rightarrow {^{\sigma}}E$ is given by
$$(X,Y) \mapsto (-\frac{Y^2}{2X^2},- \frac{\sqrt{-2}}{4}\frac{Y}{X^2}(-\alpha f_1 + X^2)).$$
This shows that $E$ is a $k$-curve with $K^{+}(\sqrt{-2})$ as a field of complete definition.

\section{Modularity}
\label{bgmod}

In the previous sections, in order to apply level lowering, we assumed modularity of the Frey curves. We will now establish this assumption. Indeed, we will apply modularity lifting theorems to establish a general modularity statement for elliptic curves from which modularity of the Frey curves will follow as a corollary.\\

It is known that all elliptic curves over $\Q$ are modular and is expected the same to be true over totally real number fields but there are no complete general results in the latter situation. Thus, \textit{a priori}, we can only conjecture modularity of the Frey curves $E_{(a,b)}^{(k_1,k_2,k_3)}$, $E_{(a,b)}^{(k_1,k_2)}$ and $E_{(a,b)}^{(k_1,n_2,\pm)}$. We already know modularity for specific Frey curves. For example, for $r=7$ and triple $(k_1,k_2,k_3)=(1,2,3)$ the corresponding Frey curve is defined over $\Q$ (see section \ref{grau7}) hence it is modular from classical results; for $r=13$ and triple $(1,3,4)$ the Frey curves are defined over $\Q(\sqrt{13})$, are not $\Q$-curves, and their modularity is a consequence of the following theorem proved in \cite{DF2}.

\begin{tm} Let $F$ be a totally real number field and $\mathcal{C}$ and elliptic curve defined over $F$. Suppose that $3$ splits completely in $F$ and $\mathcal{C}$ has good reduction at the primes above $3$. Then $\mathcal{C}$ is modular if:
\begin{itemize}
 \item $\bar{\rho}_{\mathcal{C},3}$ is absolutely irreducible, or
 \item $F$ is abelian and $\bar{\rho}_{\mathcal{C},3}$ is absolutely reducible.
\end{itemize}
\end{tm}

Modularity of all $E_{(a,b)}^{(k_1,k_2,k_3)}$ follows if the previous theorem holds with 3 being unramified on $F$ instead of totally split. A closer look at its proof tells us that the main obstacle to its generalization is guaranteeing the existence of a modular lifting of $\bar{\rho}_{\mathcal{C},3}$ that is ordinary at $v \mid 3$ exactly when $\rho_{\mathcal{C},3}$ is. Recent work due to Breuil-Diamond \cite{BreuilDiamond} building on \cite{GGL2} allows to find an adequate lifting in the case that $\bar{\rho}_{\mathcal{C},3}|G_{F(\sqrt{-3})}$ is absolutely irreducible, with no further conditions of $F$, which allow us to prove the required generalization. Actually, we can prove the following stronger theorem.

\begin{tm} Let $F$ be a totally real field and $C/F$ an elliptic curve. Then, $C$ is modular, if:
\begin{enumerate}
 \item $\bar{\rho}_{C,3}$ is irreducible and $\bar{\rho}_{C,3}|G_{F(\sqrt{-3})}$ is absolutely irreducible.
 \item $\bar{\rho}_{C,3}$ is irreducible and $\bar{\rho}_{C,3}|G_{F(\sqrt{-3})}$ reducible, $3$ unramified in $F$ and $C$ is semistable at all $v \mid 3$.
 \item $\bar{\rho}_{C,3}$ reducible, $F$ abelian, $3$ unramified in $F$ and $C$ is semistable at all $v \mid 3$.
\end{enumerate}
\label{modularidadeordinaria}
\end{tm}

\begin{proof} Write $\rho:= \rho_{C,3}$ and $\bar{\rho}:= \bar{\rho}_{C,3}$. Note that, when irreducible, $\bar{\rho}$ is modular by the theorem of Langlands-Tunnell (see Lemma 4.2 in \cite{DF2} for details). We now consider the each case separately.\\ 

(1) Suppose $\bar{\rho}$ and $\bar{\rho}|G_{F(\sqrt{-3})}$ both to be absolutely irreducible. We now apply Theorem 3.2.2 in \cite{BreuilDiamond} to $\bar{\rho}$. Let $S$ be the set of primes of bad reduction of $C$ together with the primes dividing $3$. Let also $T$ be the set of primes $v \mid 3$ such that $C$ has potentially good ordinary reduction or potentially multiplicative reduction. For each $v \in S$ set $\rho_v:=\rho|D_v$ and let $[r_v,N_v]$ be its Weil-Deligne type. Thus, for each $v \in S$, $\rho_v$ is a lifting of $\bar{\rho}|_{Gal(\bar{F}_v/F_v)}$ satisfying the hypothesis of the theorem, hence there is a global modular lifting of $\bar{\rho}$ with the local properties of $\rho$. Moreover, the theorem also states that every such lifting is modular, so $\rho$ is modular.\\

(2) Suppose $\bar{\rho}$ to be irreducible and $\bar{\rho}|G_{F(\sqrt{-3})}$ absolutely reducible. This means that the image of $\mathbb{P}(\bar{\rho})$ is Dihedral. From Lemma 9.1 in \cite{JMano} it follows that all the primes above 3 of good reduction of $C$ are primes of good ordinary reduction. In particular, since $C$ is semistable at $3$, for all $v\mid 3$ we have
\begin{equation}
\label{mord}
\bar{\rho}|D_v \sim \begin{pmatrix} \epsilon_{1}^{v} & * \\ 0 & \epsilon_{2}^v \end{pmatrix}, 
\end{equation} 
with $\epsilon_{1}^v| I_v = \bar{\chi}_3$ and $\epsilon_{2}^v$ unramified.\par

Let $\iota : \mbox{GL}_2(\mathbb{F}_3) \rightarrow \mbox{GL}_2(\mathbb{Z}[\sqrt{-2}]) \subset \mbox{GL}_2(\mathbb{C})$ be the injective group homomorphism defined by
\begin{eqnarray*}
\begin{pmatrix} -1 & 1 \\ -1 & 0 \end{pmatrix} & \mapsto & \begin{pmatrix} -1 & 1 \\ -1 & 0 \end{pmatrix} \\ 
\begin{pmatrix} 1 & -1 \\ 1 & 1 \end{pmatrix} & \mapsto & \begin{pmatrix} 1 & -1 \\ -\sqrt{-2} & -1+\sqrt{-2} \end{pmatrix}. \\
\end{eqnarray*}
The composition $\rho_0 = \iota \circ \bar{\rho} : G_F \rightarrow \mbox{GL}_2(\mathbb{C})$ is totally odd, irreducible with Dihedral projective image. Moreover, $\rho_0$ is induced by a character $\psi$ of $G_{F(\sqrt{-3})}$, thus its Artin $L$-function is entire and satisfies a functional equation, because 
$$L(\rho_0,s) = L(\mbox{Ind}_{G_{F(\sqrt{-3})}}^{G_F}\psi,s) = L(\psi,s).$$
Furthermore, given a character $\omega : G_F \rightarrow GL_2(\mathbb{C})$ we have 
$$\bar{\rho} \otimes \omega = (\mbox{Ind}_{G_{F(\sqrt{-3})}}^{G_F} \psi) \otimes \omega = \mbox{Ind}_{G_{F(\sqrt{-3})}}^{G_F}(\psi \omega)$$
and as above we conclude that $L(\rho_0 \otimes \omega,s)$ is entire. By the Converse theorem for $GL_2$ (Theorem 11.3 in \cite{JL}) we conclude that $\rho_0$ is $\mbox{GL}_2$ automorphic. Let $\pi_1$ be an automorphic form of $GL_2(F)$ of weight 1 such that $L(\rho_0,s) = L(\pi_1,s)$.\par 
The next argument using the local factors of the $L$ functions is due to Ellenberg (see Lemma 3.3 in \cite{ell3}). Let $v$ be a prime in $F$ dividing 3. On one hand, a small computation shows that the image of $\rho_0 |D_v$ is upper triangular with eigenvalues contained in $\{\pm 1\}$. Furthermore, $I_v$ fixes pointwise the subspace $V$ of dimension 1 generated by $(\sqrt{-2} + 1 , 2)$. Let $M/F$ be the extension cut out by $\bar{\rho}$ and $v_0 \mid v$ a prime in $M$. 
The local factor $L_v(\rho_0,s)$ is equal to
$$\frac{1}{\mbox{det}(\mbox{Id} - \mbox{Nm}(v)^{-s}\rho_0(\mbox{Frob}_{v_0})|V)} = (1-a_v(\rho_0)\mbox{Nm}(v)^{-s})^{-1},$$
where $a_v(\rho_0)$ must be $\pm 1$. 
On the other hand, let $f_1$ be a Hilbert modular newform of parallel weight 1 associated with $\pi_1$ and $c(v, f_1)$ the eigenvalue of the action of the Hecke operator $T_v$ acting on $f_1$. Thus we have 
$$L_v(f_1,s) = (1 - c(v,f_1)\mbox{Nm}(v)^{-s})^{-1}$$ 
and from $L_v(\rho_0,s) = L_v(f_1,s)$ we conclude $c(v,f_1) = \pm 1$. In particular, $c(v,f_1) \not\equiv 0$ (mod $v$) for all $v \mid 3$ and $f_1$ is ordinary (at 3) in the sense of \cite{wiles2}. From Theorem 1.4.1 in \cite{wiles2} we conclude that $f_1$ belongs to a Hida family $\mathcal{F}$ and we let $f_2$ be its specialization to weight 2. From Theorem 2.1.4 in \cite{wiles2} we have associated to $f_2$ a Galois representation $\rho_{f_2, \lambda}$ of $G_F$, ordinary at all $v \mid 3$, such that $\bar{\rho}_{f_2, \lambda} \sim \bar{\rho}$.\par 
At this point we have that: det$(\rho) = \chi_3$, $\rho$ is ordinary at all $v \mid 3$ (because $C/F$ is of ordinary or multiplicative reduction at all $v \mid 3$), $\bar{\rho}$ is absolutely irreducible and $D_v$-distinguished for all $v \mid 3$, and $\rho_{f_2, \lambda}$ is an ordinary lift of $\bar{\rho}$. Thus we conclude from Theorem 5.1 in \cite{SW2} that $\rho$ is modular.\\

(3) Suppose that $\bar{\rho}$ is reducible and $F$ abelian and 3 unramified at $F$. We want to apply Theorem A in \cite{SW1}. The hypothesis there are not numerated but we will refer to them as conditions (1) to (5) in descending order.\par 

As in part 2), from Lemma 9.1 in \cite{JMano} it follows that all the primes above 3 are of good ordinary or bad multiplicative reduction, hence condition (4) holds. Furthermore, we have $\bar{\rho}^{ss} = \chi_1 \oplus \chi_2$, where $\chi_1 = \psi \bar{\chi}_3$, $\chi_2 = \psi^{-1}$ where $\psi$ is ramified only at additive primes of $C$. Note also that $\psi$ must be quadratic because $\mathbb{F}_3^{*}$ has only two elements. Then $\chi_1 / \chi_2 = \psi^2 \bar{\chi}_3 = \bar{\chi}_3$ and conditions (2) and (3) are satisfied. Finally, the extension $F(\chi_1 / \chi_2) = F(\sqrt{-3})$ of $\Q$ is abelian because $F$ and $\Q(\sqrt{-3})$ are both abelian. This establishes condition (1) and condition (5) holds because $\rho$ arises from an elliptic curve. Thus by Theorem A in \cite{SW1} we conclude that $\rho$ is  modular.\end{proof}

Since $\bar{\rho}_{E,3}$ must satisfy one of the cases in the statement of the previous theorem the following useful theorem follows immediately.

\begin{tm} Let $F$ be an abelian totally real number field where $3$ is unramified. Let $C/F$ be an elliptic curve semistable at all primes $v \mid 3$. Then, $C/F$ is modular.
\end{tm}

\begin{cor} The Frey curves $E_{(a,b)}^{(k_1,k_2,k_3)}$, $E_{(a,b)}^{(k_1,k_2)}$ and $E_{(a,b)}^{(k_1,n_2,\pm)}$ are modular.
\label{cor:modularity}
\end{cor}
\begin{proof} The Frey curves are defined over the abelian totally real field $K^+ \subset \Q(\zeta_r)$. Since $r \geq 7$ we have that 3 is unramified in $K^+$.
From the last statements on Propositions \ref{condEgamma}, \ref{condEgammaII} and \ref{cond222} we have that all the Frey curves are semistable at all $v \mid 3$.
The result now follows from the theorem.
\end{proof}

\section{Irreducibility of $\bar{\rho}_{E,p}$}
\label{sec:irr}

As discussed, we need irreducibility of $\bar{\rho}_{E,p}$ in order to apply the level lowering theorems. In this section, we prove the existence of a bound $M_r$, for each fixed $r \geq 7$, such that if $p > M_r$ then $\bar{\rho}_{E,p}$ is irreducible for all the Frey curves constructed. We achieve this by applying the results in \cite{FS} to setting.

\begin{tm} Let $r \geq 7$ be a fixed prime and let $(a,b,c)$ be a primitive solution $x^r + y^r =Cz^p$. Write $E$ for $E_{(a,b)}^{(k_1,k_2,k_3)}$ or $E_{(a,b)}^{(k_1,k_2)}$ or $E_{(a,b)}^{(k_1,n_2,\pm)}$. Then, there exists a computable constant $M_r$ such that, if $p > M_r$ then the representation $\bar{\rho}_{E,p}$ is absolutely irreducible.
\label{irr}
\end{tm}
\begin{proof} The curves $E$ are defined over $k$, $K_0$ or $K^+$ which are totally real. Let $h_r$ be the class number of $K^+$ and $d_r$ its degree. Let $B_r$ be the constant $B$ in the statement of Theorem~1 in \cite{FS}, when taking $K=K^+$. It is clear that $B_r$ only depends on $r$. Write $I_r = (1 + 3^{d_r h_r})^2$. Then, if $p \nmid B_r$ and $p > I_r$ it follows from Theorem~2 in \cite{FS} that $\bar{\rho}_{E,p}$ is irreducible. The theorem follows by choosing $M_r$ greater than $B_r$ and $I_r$.
\end{proof}

The constant $M_r$ in the previous theorem work for all Frey curves at once, but when working with particular examples we are interested in finding the smallest possible constant. We can get a smaller value for $M_r$ when
\begin{itemize}
 \item working with Frey curves over $k$ or $K_0$. Using the corresponding values of $B$, degree and class number will bring down the constant with exactly the same proof;
 \item working with particular values of $r$ and the Frey curves $E = E_{(a,b)}^{(k_1,k_2,k_3)}$, because we can use the fact that we know explicit primes of good reduction. For example, in section 6 of \cite{FS} this allowed to bring the constant down to 17 in the particular case of $r=13$.
\end{itemize}

In the previous proof we used very little specific information about the Frey curves. In the next theorem, we particularize to $E = E_{(a,b)}^{(k_1,k_2,k_3)}$ and we are able to prove a much better bound in some cases. For a prime $\p$ of potentially good reduction of $E$ define $\Phi_{\p}$ as in the introduction of \cite{kraus}. Since $K^+$ is Galois the inertial degree $f(r) = f(\mathfrak{P}_2 / 2)$ is the same for all primes $\mathfrak{P}_2$ above 2. 

\begin{tm} Fix $r \geq 7$ a prime and $(k_1,k_2,k_3)$ an $r$-suitable triple. The Frey curve $E = E_{(a,b)}^{(k_1,k_2,k_3)}$ is defined over $K^+$ or $K_0$. Let $\mathfrak{P}_2 \mid 2$ be a prime in $K^+$ in the former case or in $K_0$ in the latter case. Suppose that the inertial degree $f(\mathfrak{P}_2 / 2)$ is odd. Let $(a,b,c)$ be a primitive solution of $(\ref{eq1})$ or $(\ref{eq2})$. Then $\bar{\rho}_{E,p}$ is irreducible for all primes $p \geq 3$.
\label{irr2}
\end{tm}
\begin{proof} Recall that the denominator of $j(E)$ is $(\alpha\beta\gamma f_{k_1}f_{k_2}f_{k_3})^2$, hence $\upsilon_{\mathfrak{P}_2}(j(E)) \geq 0$ then $E$ has potentially good reduction at $\mathfrak{P}_2$. 
Also, $\upsilon_{\mathfrak{P}_2}(\Delta) = 4 \not\equiv 0$ (mod 3), where $\Delta$ is minimal at $\mathfrak{P}_2$, then we are in case (ii) of Theorem 3 in \cite{kraus}. Then, from the same Theorem 3 we conclude $|\Phi_{\mathfrak{P}_2}|=3,6,24$. Moreover, $2^{nf}(2^f -1)$ is divisible by $3$ only if $f$ is even. Since we have $f(r)$ odd we apply Proposition 3.3 in \cite{bil2} to conclude that $\bar{\rho}_{E,p}$ is irreducible for all primes $p \geq 3$.\end{proof}

\section{Equations of signature $(7,7,p)$, Part I.}
\label{sete}

In this section we will use Frey curves constructed in section~\ref{curvasI} to show the non-existence of non-trivial primitive first case solutions to equations of the form $x^7 + y^7 = C z^p$ for infinitely many values of $C$. More precisely, we will prove the following theorem.

\begin{tm}
\label{grau7}
Write $d=2^{s_0}3^{s_1}5^{s_2}$. Let $\gamma$ be an integer only divisible by primes $l \not\equiv 1,0 \mbox{ (mod } 7)$. Let $p \geq 17$ be a prime. We have that: 
\begin{itemize}
\item[$(I)$] The equation $x^7 + y^7 = d \gamma z^p$ has no non-trivial first case solutions if
any of the following cases is satisfied:
 \subitem 1) $s_0 \geq 2$, $s_1 \geq 0$ and $s_2 \geq0$;
 \subitem 2) $s_0 = 1$, $s_1 \geq 1$ and $s_2 \geq 0$;
 \subitem 3) $s_0 = 0$, $s_1 \geq 0$ and $s_2 \geq 1$.
\item[$(II)$] The equation $x^{14} + y^{14} = d \gamma z^p$ has no non-trivial primitive solutions if $s_1 > 0$ or $s_2 > 0$ or $s_0 \geq 2$. 
\end{itemize}
\end{tm}

Suppose that $(a,b,c) \in \mathbb{Z}^3$ is a non-trivial primitive solution to $x^7 + y^7 =  d\gamma z^p$. We will use the Frey curves \eqref{Freycurve} of section \ref{curvasI}. Note that $(1,2,3)$ is the only $7$-suitable triple. Since $7=6 + 1$, from Proposition~\ref{isos} we conclude that the Frey curve $E:=E_{(a,b)}^{(1,2,3)}$ has a model over $K_0=\Q$. Indeed, after applying the recipe in section \ref{descend} we obtain a Frey curves over $\Q$ with the form 
\begin{eqnarray*}
   E_{(a,b)} & : & Y^2 = X^3 + a_4 X + a_6, \\
   a_4 & = & -3024(a^4 - a^3 b + 3a^2 b^2 - ab^3 + b^4) \\
   a_6 & = & 12096(a^6 - 15a^5 b + 15a^4 b^2 - 29a^3 b^3 + 15a^2b^4 - 15ab^5 + b^6).\
\end{eqnarray*}

\begin{rem} These curves were also found by Kraus \cite{kraus1} and Dahmen \cite{Dahm}.
\end{rem}

We also know from the discussion in section~\ref{curvasI} that we can assume $d\gamma \mid a+b$. It will became clear in the sequel that for the proof of Theorem \ref{grau7} we actually only need $d \mid a+b$.

\begin{pp} The curves $E_{(a,b)}$ have conductor given by
$$ N_E = \left\{ \begin{array}{ll}
2^{2}7^2\mbox{rad}(c) \mbox{ or } 2^{3}7^2\mbox{rad}(c) & \quad \mbox{ if } 2 \nmid a+b\\ 
2^{4}7^2\mbox{rad}(c) & \quad \mbox{ if }  2 \parallel a+b \\
2^{3}7^2\mbox{rad}(c)  & \quad  \mbox{ if }  4 \mid a+b \\
\end{array} \right.$$
Moreover, if $2 \nmid a+b$ we can suppose that $a$ is even and the conductor is 
$$ N_E = \left\{ \begin{array}{ll}
2^{2}7^2\mbox{rad}(c)  & \quad \mbox{ if } 4 \mid a\\ 
2^{3}7^2\mbox{rad}(c) & \quad \mbox{ if }  4 \nmid a \\
\end{array} \right.$$
\label{conductor}
\end{pp}
\begin{proof} From Proposition \ref{cond3} we know the set of possible values for the conductor. With the help of SAGE we compute the values of the conductor for all pairs $(a,b)$ mod $2^8$ and observe how they relate to $a+b$.\end{proof}

\subsection{Irreducibility and Level lowering}

The Frey curves are defined over $\Q$ then we can use Serre's conjecture. Since $(a,b,c)$ is non-trivial there exists a prime $q > 6$ dividing $c$, i.e. of multiplicative reduction for $E_{(a,b)}$. Then if $p \geq 17$ we have that $\bar{\rho}_{E,p}$ is absolutely irreducible (see Theorem 22 in \cite{Dahm}). Moreover, we know that $\bar{\rho}_{E,p}$ is odd and finite at $p$. Proposition~\ref{artinc2I} tells us that the Artin conductor (outside $p$) $N(\bar{\rho}_{E,p})$ is $2^{s}7^2$, where $s=2,3$ or 4. Furthermore, the exponent $s$ relates to $a+b$ as in Proposition~\ref{conductor}.\\

Let $S_2(M)$ denote the set of cuspforms of weight 2, trivial character  and level $M$. By Serre's conjecture there must exist a newform $f \in S_2(N(\bar{\rho}_{E,p}))$ and a prime $\mathfrak{P} \mid p$ in $\bar{\Q}$ such that 
\begin{equation}
\bar{\rho}_{E,p} \sim \bar{\rho}_{f,\mathfrak{P}}.
\label{cong}
\end{equation}

\subsection{Eliminating newforms} To finish the proof of Theorem \ref{grau7} we need to contradict (\ref{cong}). Using SAGE software we compute the newforms in $S_2(2^{s}7^2)$ with $s=2,3$ or 4 and we divide them into two sets
\begin{itemize}
 \item [S1:] Newforms with $\Q_f = \Q$
 \item [S2:] Newforms such that $\Q$ is strictly contained in $\Q_f$  
\end{itemize}
In the sequel we will find a contradiction to (\ref{cong}) for each newform in both sets, using slightly different arguments for each set.\\

To get a contradiction to \eqref{cong} we will need to know some traces of Frobenius of our Frey curves. Let $\ell$ be a prime. Any solution $(a,b,c)$ will satisfy $(a,b) \equiv (x,y) \pmod{\ell}$ with $0 \leq x,y \leq \ell-1$. Then, fixed $\ell$, there is only a finite number of residual curves modulo $\ell$ and, using a computer, we can easily determine all the possible values for $a_{\ell}(E) = \ell + 1 - \tilde{E}_{(x,y)}(\F_\ell)$. In particular, for $\ell \in  \{3,5,23 \}$ we obtained the following results 
$$\left\{ \begin{array}{ll}
 a_3(E)    & \in \{-1 , 3\}, \\
 a_5(E)    & \in \{-3, -1 ,1, 3\}, \\ 
 a_{23}(E) & \in \{-9, -7 ,-5, -1, 1, 3\}. \
\end{array} \right.$$
Moreover,
\begin{equation}
a_3(E_{(a,b)}) = -1 \quad \mbox{ if } 3 \mid a+b  \quad \mbox{ and } \quad a_5(E_{(a,b)}) = - 1 \mbox{ if } \quad 5 \mid a+b.
\label{condextra}
\end{equation} 

We now proceed to eliminate the newforms.\\ 

\textsc{Newforms in S1:} Given a newform $f$ in S1 we want to find a prime $\ell$ such that $a_\ell(f)$ is not in the corresponding set above, because this will give a contradiction if $p$ is large enough. Indeed, as long as $p > 7$, by comparing the coefficients $a_3(f)$ and $a_{23}(f)$ of all the newforms in S1 against the values in the previous sets we find a contradiction to the isomorphism (\ref{cong}) for all $f$ in S1 except for the newforms corresponding to the curves $E_{(0,1)}$, $E_{(1,-1)}$ and $E_{(1,1)}$. As discussed in section~\ref{limitations}, these newforms were expected to survive because $(0,1,1)$, $(1,1,1)$ and $(1,-1,1)$ are trivial solutions of equations (\ref{eq1}) and (\ref{eq2}).\\

Observe that the conductors at 2 of $E_{(0,1)}$, $E_{(1,-1)}$ and $E_{(1,1)}$ are $2^2$, $2^3$ and $2^4$, respectively. Recall that $d=2^{s_0}3^{s_1}5^{s_2}$. Thus, from Proposition~\ref{conductor} we have that 
\begin{itemize}
 \item if $s_0 \geq 0$ then it remains to eliminate $E_{(0,1)}$, $E_{(1,-1)}$, $E_{(1,1)}$;
 \item if $s_0 = 1$ then it remains to eliminate $E_{(1,-1)}$, $E_{(1,1)}$;
 \item if $s_0 \geq 2$ then it remains to eliminate only $E_{(1,-1)}$.
\end{itemize}
Conditions of the form $d \mid a+b$ impose restrictions on the traces of Frobenius at primes $q \mid d$. In particular, we will now use \eqref{condextra} to eliminate the newforms attached with $E_{(1,1)}$ and $E_{(1,0)}$.

\begin{rem}Note that no condition of the form $d \mid a+b$ is able to eliminate $E_{(1,-1)}$, because $d \mid (1 + (-1)) = 0$.
\end{rem}

It easy to check 
$$a_3(E_{(0,1)}) =-1,\quad a_3(E_{(1,-1)})=-1, \quad a_3(E_{(1,1)}) = 3.$$ 
Suppose now that $s_1 \geq 1$. Hence $3 \mid a+b$ and from \eqref{condextra} we have $a_3(E_{(a,b)})=-1$. This eliminates $E_{(1,1)}$ as long $p > 3$. Moreover, together with $s_0 = 1$ we are left only with $E_{(1,-1)}$. Note also that
$$a_5(E_{(0,1)}) =-3,\quad a_5(E_{(1,-1)})=-1, \quad a_5(E_{(1,1)}) = 1.$$ 
Suppose now that $s_2 \geq 1$. Hence $5 \mid a+b$ and from \eqref{condextra} we have $a_5(E_{(a,b)})=-1$. This eliminates both $E_{(1,0)}$ and $E_{(1,1)}$ as long $p > 3$ and we are left only with $E_{(1,-1)}$.\\

To eliminate the newform corresponding to $E_{(1,-1)}$ we will use the inertia at 7 by following Kraus \cite{kraus}. Let $C/\Q_7$ be an elliptic curve and $\Phi_7(C)$ be the Galois group of the extension (of the maximal unramified extension of $\Q_7$) where $C$ acquire good reduction at 7. From Proposition 1 in \cite{kraus} we see that 
$$|\Phi_7(C)| = \mbox{ denominator of }(\frac{\upsilon_7(\Delta_{\mbox{min}}(C))}{12}),$$
hence $|\Phi_7(E_{(a,b)})| = 3$ or 6, if $7 \mid a+b$ or $7 \nmid a+b$, respectively. In particular $|\Phi_7(E_{(1,-1)})| = 3$  and (\ref{cong}) cannot hold if $7 \nmid a+b$, because the inertia at 7 will not match if $p > 7$.\\ 

We now summarize the previous discussion: we have eliminated all the newforms in S1 if we assume that $(a,b,c)$ is a non-trivial primitive first case solution and
\begin{itemize}
 \item $s_0 \geq 2, s_1 \geq 0$ and $s_2 \geq 0$ or,
 \item $s_0 \geq 1, s_1 \geq 1$ and $s_2 \geq 0$ or,
 \item $s_0 \geq 0, s_1 \geq 0$ and $s_2 \geq 1$.
\end{itemize}
To finish the proof of part (I) of Theorem~\ref{grau7} we have to eliminate the newforms in S2.\\

\textsc{Newforms in S2:} Suppose that (\ref{cong}) holds for $f = q + \sum_{n\geq2} {c_n(f)}q^n$ in S2. In particular, the congruence 
\begin{equation}
a_3(E) \equiv c_3(f) \mbox{ (mod } \mathfrak{P}) 
\end{equation}
also holds for some prime $\mathfrak{P}$ in $\bar{\Q}$ above $p$. This is not possible if $p > 7$. Indeed, for all newforms $f$ in S2 the minimal polynomial of the Fourier coefficient $c_3(f)$ is $x^2 - 2$ or $x^2 - 8$ then, for example, in the latter case
we must have $$ 0 \equiv c_3^2 -8 \equiv a_3^2 - 8 \mbox{ (mod } p).$$  
Since our curves verify $a_3 \in \{-1,3\}$ the previous congruence implies that $0 \equiv -7, 1 \mbox{ (mod } p)$ which is impossible if $p > 7$. The same holds with the other minimal polynomial and this concludes the proof of part (I) of Theorem \ref{grau7}.\\

\subsection*{Proof of part (II) of Theorem~\ref{grau7}.} Recall that $d=2^{s_0}3^{s_1}5^{s_3}$ and suppose that $(a,b,c_0)$ is a non-trivial primitive solution to $x^{14} + y^{14} = dz^p$. Observe that we have the factorization
$$a^{14} + b^{14} = (a^2 + b^2)\phi_7(a^2,b^2) = dc_0^p$$ 
and also that $d \mid a^2 + b^2$. By looking modulo 3 and 4 we find that $a^2 + b^2$ with $(a,b)=1$ is never divisible by 3 or 4. Thus for $s_0 \geq 2$ or $s_1 > 0$ Theorem \ref{grau7} (II) immediately holds. We are left to deal with the exponent $s_2$. By looking modulo 7 we find that $a^2 + b^2$ is never divisible by 7 then the solution $a^{14} + b^{14} = (a^2 + b^2)\phi_7(a^2,b^2) = dc_0^p$ will correspond to a solution $(a,b,c)$ of the equation
\begin{equation}
\phi_7(a^2,b^2) = c^p \quad \quad  \mbox{with } \quad d \mid a^2+b^2
\label{eqq}
\end{equation}
Given a primitive solution $(a,b,c)$ of (\ref{eqq}) we attach to it $E = E_{(a^2,b^2)}$ as a Frey curve. From the fact $4 \nmid a^2 + b^2$, Proposition \ref{conductor} and Serre's conjecture it follows that there exist a newform $f \in S_2(M)$ with $M=2^2 7^2$ or $2^4 7^2$ satisfying $\bar{\rho}_{E,p} \sim \bar{\rho}_{f,\mathfrak{P}}$. We do as above and divide the newforms into the same sets S1 and S2. Since the newform associated with the solution $(1,-1,0)$ has level $2^3 7^2$ it will not belong to S1 nor to S2. Hence the restriction $7 \nmid c$ is not needed. If $s_2 > 0$ then $5^{s_2} \mid a^2 + b^2$ and we have $a_5(E_{(a^2,b^2)}) = -1$. We already know that this condition is enough to eliminate the newforms associated with $E_{(0,1)}$ and $E_{(1,1)}$. This eliminates all the newforms in S1 and we treat those in S2 exactly as in the proof of (I).\qed

\section{Equations of signature $(7,7,p)$, Part II.}
\label{seteII}

In this section we will use the Frey curves in Section \ref{curvasII}. Since the curves are not defined over $\Q$ it is not possible to make use of classical modularity and irreducibility results over $\Q$, hence we need to apply the full strength of the methods developed here. We will prove the following statement.

\begin{tm} Let $ p > (1 + 3^{18})^2$ be a prime. Then the equation $x^7 + y^7 = 4z^p$ has no non-trivial primitive solutions. 
\label{grau7II}
\end{tm}

Note that $\Q(\zeta_7) \supset K^+ = \Q(z)$ where $z$ satisfies $z^3 - z^2 - 2z + 1 = 0$. Let $\pi_7$ be the prime in $K^+$ above 7. Suppose that $(a,b,c)$ is a non-trivial primitive solution to $x^7 + y^7 = 4 z^p$. We will use the Frey curves \eqref{FreycurveII} with $(k_1,k_2)=(1,2)$. By following the recipe in section \ref{curvasII} we obtain 
$$ \alpha = z^2 + z - 2, \quad \beta  = -z^2 + 4, \quad \gamma = -z-2 $$ 
and also
$$ A:= A_{(a,b)} = \alpha (a+b)^2 \quad \quad \mbox{ and } \quad \quad B:=B_{(a,b)} = \beta (a^2 -zab + b^2).$$
Then we have a Frey curve $E:=E_{(a,b)}^{(1,2)}$ over $K^+$ given by
\begin{equation}
E : Y^2 = X(X-A)(X+B).
\end{equation}

\begin{pp} The conductor $N_E$ of $E$ is of the form 
\[
2^{s} \pi_7^t \mbox{Rad}_{2r}((a+b)c_1c_2) \quad \mbox{ where } \quad s \in \{0,1,3,4\}
\]
and $t=1$ or $2$ if $7 \mid a+b$ or $7 \nmid a+b$, respectively. Moreover, 
\begin{enumerate}
 \item[-] if $2 \mid a+b$ then $s \in \{0,1,4\}$;
 \item[-] if $4 \mid a+b$ then $s \in \{0,1\}$;
 \item[-] if $2 \nmid a+b$ and $4 \nmid a$ then $s \in \{3,4\}$;
 \item[-] if $2 \nmid a+b$ and $4 \mid a$ then $s=3$. 
\end{enumerate}
Furthermore, $c_1,c_2 \in \OO_{K^+}$ are divisible only by primes $\q$ dividing rational primes congruent to $1 \pmod{7}$.  
\label{cond4}
\end{pp}
\begin{proof} The part of the conductor corresponding to the primes outside 2 follows immediately from Proposition \ref{condEgammaII}. For the prime 2 it is enough to check the values of the conductor for all $(a,b)$ mod $2^8$. We did this with a computer and observed how they relate to $a+b$.
\end{proof}

\subsection{Irreducibility and level lowering}

We now prove irreducibility of $\bar{\rho}_{E,p}$ for large $p$.

\begin{tm} Suppose that $(a,b,c)$ is a primitive solution to $x^7 + y^7 = 4 z^p$. Write $I_4 = (1 + 3^{18})^2$. Then, the representation $\bar{\rho}_{E,p}$ is irreducible for all $p > I_4$.
\label{irr4}
\end{tm}
\begin{proof} This is a particular case of Theorem~\ref{irr}. The field $K^+$ is cubic and has class number 1. The generating units are $\epsilon_1 = z$ and $\epsilon_2 = -z+1$. We computed the value of $B_7$ and checked that it is divisible only by the primes $7$ and $13$. The result now follows from Theorem~\ref{irr}.
\end{proof}

Let $f_E$ be the newform associated with $E$ by modularity (see Corollary~\ref{cor:modularity}). As explained in section~\ref{curvasII}, from Proposition~\ref{cond4}, Theorem~\ref{irr4} and level lowering we conclude that for all $p > I_4$ we have
\begin{equation}
\bar{\rho}_{f_E,p} \sim \bar{\rho}_{f,\fp}, \quad \quad \mbox{for some } \fp \mid p,
\label{contradictionII}
\end{equation}
where $f$ is a newform in $S_2(2^s\pi_7^t)$ with $s=0,1$ and $t=1,2$. 

\subsection{Eliminating newforms}

Let $\q=(z^2 + z - 3)\OO_{K^+}$ be a prime above 13. With Magma we computed the newforms in the spaces above. In particular, we obtained that
\begin{itemize}
 \item there are no newforms at level $\pi_7$;
 \item there is one newform  at level $\pi_7^2$ satisfying $a_\q(f)=0$;
 \item there is one newform  at level $2\pi_7$ satisfying $a_\q(f)=-4$;
 \item there are two newforms $f,g$ at level $2\pi_7^2$ satisfying $a_\q(f)=4$, $a_\q(g)=0$.
\end{itemize}

Suppose that $13 \nmid a+b$, hence $E$ has good reduction at $\q$. We computed the residual Frey curves for pairs $(a,b) \in \F_{13}^2$ such that $a+b \not\equiv 0 \pmod{13}$ and checked that $a_\q(E) \in \{-6,-2,2\}$. Since isomorphism \eqref{contradictionII} implies the congruence 
\[
a_\q(E) \equiv a_\q(f), 
\]
from the trace values $a_\q(f)$ and $a_\q(E)$ we obtain a contradiction with $p > 5$. 

Suppose now $13 \mid a+b$. Thus $\q$ is of multiplicative reduction for $E$ and, in \eqref{contradictionII} level lowering is actually happening at $\q$. This requires that 
$$a_\q(f) \equiv \pm (\mbox{Norm}(\q) + 1) \pmod{p},$$
which gives a contradiction if $p > 7$.

\section{Equations of signature $(7,7,p)$, Part III.}
\label{seteIII}

In this section we combine the Frey curves of the previous two sections to prove the following theorem.

\begin{tm} There is some constant $M_6$ such that if $p > (1 + 3^{18})^2$ and $p \nmid M_6$ then the equation $x^7 + y^7 = 6z^p$ has no non-trivial primitive solutions. 
\label{grau7III}
\end{tm}

\subsection{Using two Frey curves.} We first explain why we need to use multiple Frey curves. It is clear that one can try to prove Theorem~\ref{grau7III} with the same approach used in the proof of Theorem~\ref{grau7II}. This strategy cannot be applied due to computational limitations. Indeed, in the proof of Theorem~\ref{grau7II} it was the condition $C=4 \mid a+b$ that restricted the conductor of the Frey curves $E=E_{(a,b)}^{(2,1)}$ at 2 to be $2^0$ or $2^1$. This gave rise to very small spaces of newforms that we easily computed. For the present theorem we have $C=6$ (hence $2 \mid a+b$ but $a+b$ is not necessarily divisible by 4), allowing the conductor of $E$ at 2 to be $2^s$ with $s \in \{0,1,4\}$.

\bigskip 

For an ideal $M$ of $K^+$ let $S_2(M)$ denote the space of cuspidal Hilbert modular forms of level $M$, parallel weight 2 and trivial character. The predicted level of newforms we have to compute are of the form $2^s 3 \pi_7^t$ (see Proposition~\ref{artinc2}). In particular, when $s=4$ and $t=2$ the dimension of the cuspidal subspace is 86017 which makes it is impossible to compute newforms. To circumvent this problem we observe the following:  
\begin{itemize}
 \item If $7 \nmid a+b$ then Theorem~\ref{grau7III} follows from Theorem~\ref{grau7} for $p \geq 17$. Thus we can suppose that $7 \mid a+b$.
 \item Note that in the previous bullet, disguised in Theorem~\ref{grau7}, we use the Frey curve $E_{(a,b)}^{(1,2,3)}$. We know these are not good for the case $7 \mid a+b$, hence in everything that follows we will work with $E:=E_{(a,b)}^{(2,1)}$ over $K^+$.
 \item From Proposition~\ref{cond4} we see that $t=1$ if and only if $7 \mid a+b$. Thus the hardest level one needs to compute is $M=2^4 3 \pi_7$. The dimension of $S_2(M)$ is 12289 which still makes its full computation impossible.
 \item Nevertheless, using algorithms implemented in Magma \cite{Magma} for quaternion algebras due to Kirschmer and Voight \cite{KVoight} and Hilbert modular forms due to Demb\'el\'e \cite{Dembele}, John Voight was able to compute the rational newforms of level $M$. More precisely, he computed their Fourier coefficients for all primes $\q$ with norm up to 200. The output can be found at \url{http://www.ub.edu/tn/visitant/amat.php}.
 \item At the cost of imposing further restrictions on the exponent $p$, it is possible to eliminate all the newforms with non-rational coefficients without computing them. Thus, we can prove Theorem~\ref{grau7III} if we manage to eliminate all the rational newforms computed by J. Voight.
\end{itemize}

\subsection{Irreducibility and Level lowering} The following theorem follows also from the proof of Theorem~\ref{irr4}.

\begin{tm} Suppose that $(a,b,c)$ is a primitive solution to $x^7 + y^7 =6z^p$. Write $I_6 = (1 + 3^{18})^2$. Then, the representation $\bar{\rho}_{E,p}$ is irreducible for all $p > I_6$.
\label{irr6}
\end{tm}
Suppose that $(a,b,c)$ is a primitive solution to $x^7 + y^7 =6z^p$ with $p > I_6$. We can assume that $7 \mid a+b$. Let $f_E$ be the newform associated with $E$ by modularity (see Corollary~\ref{cor:modularity}). As before, from Theorem~\ref{irr6}, Proposition~\ref{cond4} and level lowering we conclude that 
\begin{equation}
\bar{\rho}_{f_E,p} \sim \bar{\rho}_{f,\fp}, \quad \quad \mbox{for some } \fp \mid p,
\label{contradictionIII}
\end{equation}
where $f$ is a newform in $S_2(2^s 3 \pi_7)$ with $s \in \{0,1,4\}$.

\subsection{Eliminating newforms} In this case, contradicting isomorphism \eqref{contradictionIII} for all predicted $f$ requires more work. We start by explaining a useful trick that we use in the sequel and that can easily be adapted for other Diophantine applications.

\subsection*{Bounding the exponent using splitting primes} Previously, we used the fact that the Frey curves $E$ have good reduction at certain primes $\q$ and so we can compute all the residual curves and determine all the possible values for $a_\q(E)$. We will now explain how we can use splitting primes of good reduction to take this idea further and obtain better bounds for the exponent.

\bigskip

Let $q \not\equiv 1 \pmod{7}$ be a rational prime that splits in $K^+$. Write $\q_1,\q_2,\q_3$ for the primes dividing $q$. For a solution $(a,b,c) \in \Z^3$ we have $(a,b) \equiv (x,y) \pmod{\q_i}$ for $i=1,2,3$ with $0 \leq x,y \leq q-1$ independent of $i$. Thus, fixed $(x,y) \pmod{q}$ such that $x+y \not\equiv 0 \pmod{q}$ we can compute the triple $(a_{\q_1}(E),a_{\q_2}(E),a_{\q_3}(E))$. This way, instead of only knowing the possible values of $a_{\q_i}(E)$ at each prime we know how they group together.

Suppose that $q \nmid a+b$, hence $E$ has good reduction at all $\q_i$. By computing \eqref{contradictionIII} at $\mbox{Frob}_{\q_i}$ and taking traces we get
$$a_{\q_i}(E) \equiv a_{\q_i}(f) \pmod{\fp}.$$ 
Suppose $q \mid a+b$, hence $E$ has multiplicative reduction at all $\q_i$. Thus, there is level lowering happening in \eqref{contradictionIII} at $\q_i$ and we must have
$$a_{\q_i}(f) \equiv \pm (q + 1) \pmod{\fp}.$$
Let $f$ be a newform with field of coefficients $K_f$. Write $\Norm$ for the usual norm of $K_f / \Q$. Suppose that $f$ satisfies \eqref{contradictionIII}. Then, we can define
$$A_{x,y}(f) := \begin{cases}
   \mbox{gcd}\{\Norm(a_{\q_i}(E) - a_{\q_i}(f))\} \quad \mbox{ if } x+y \not\equiv 0 \pmod{q} \\
   \mbox{gcd}\{\Norm((q + 1)^ 2 - a_{\q_i}(f)^2)\} \quad \mbox{ if } x+y \equiv 0 \pmod{q} \
\end{cases}$$
and also 
\begin{equation}
 B_{q}(f) := \prod_{\substack{(x,y) \in \F_{q}^{2}\\ (x,y)\neq(0,0)}}{A_{x,y}(f)}.
\end{equation}
We have now proved.
\begin{tm} Suppose that $(a,b,c)$ is a primitive solution to $x^2 + y^2 =6z^p$ for $p \geq I_6$. Then, $7 \mid a+b$ and \eqref{contradictionIII} holds with a newform $f \in S_2(2^s 3 \pi_7)$ with $s \in \{0,1,4\}$. Moreover, $p \mid B_{q}(f)$.
\label{bound}
\end{tm}

\subsection*{Newforms with rational coefficients.}\label{subsec:rational} We will now use Theorem~\ref{bound} to eliminate newforms with rational coefficients. Note that 13 splits in $K^+$ and write $\q_1$, $\q_2$, $\q_3$ for the ideals dividing it, which are given by
$$ \q_1 = (z^2 + z - 3)\OO_{K^+}, \quad  \q_2 = (-z^2 + 2z + 2)\OO_{K^+},  \quad \q_3 = (2z^2 - z - 2)\OO_{K^+}.$$
Using Magma we computed the newforms in the spaces $S_2(2^s 3 \pi_7)$ for $s=0,1,2$ and obtained,
\begin{itemize}
 \item for $s=0$: one rational newform $f_1$ satisfying $a_{\q_i}(f_1) =-2$ for $i=1,2,3$;
 \item for $s=1$: one rational newform $f_2$ satisfying $a_{\q_i}(f_2) = 6$ for $i=1,2,3$; four conjugacy classes of newforms with $a_{\q_1}$ in a cubic field. In particular, there is one $g$ such that $a_{\q_1}(g)$ satisfies the polynomial $x^3 - 7x^2 + 10x + 7$.
 \item for $s=2$: 8 rational newforms and 11 conjugacy classes of non rational newforms;
\end{itemize}
Moreover, for each rational newform $f$ above we computed the quantity $B_{13}(f)$. All these numbers turn out to be non-zero for all rational $f$ and, the only prime factors showing up are $2,3,5$. 

\begin{rem} We give an example to clarify the advantages of using splitting primes. By going through all the pairs $(a,b) \in \F_{13}^2$ such that $a+b \not\equiv 0 \pmod{13}$, one easily computes that the Frey curves satisfy
$$a_{\q_1}(E) \in \{ -6, -2, 2 \}, \quad a_{\q_2}(E) \in \{ -6, -2, 2 \}, \quad a_{\q_3}(E) \in \{ -6, -2, 2, 6 \}.$$
In particular, $a_{\q_i}(E)$ can take the value $-2$ for any $i$. Nevertheless, it is easy to check with the computer that the triple $(a_{\q_1}(E),a_{\q_2}(E),a_{\q_3}(E))$ never takes the value $(-2,-2,-2)$. Thus, for the newform $f_1$ it is clear that $B_{13}(f_1) \neq 0$.
\end{rem}

For the remaining level $2^4 3 \pi_7$ we used the list of 462 rational newforms provided by John Voight. After computing $B_{13}(f)$ for all these forms we obtained a sublist $\mathcal{L}$ consisting of 32 newforms such that $B_{13}(f) = 0$. To eliminate those in $\mathcal{L}$ we apply Theorem~\ref{bound} with the prime $q=41$. Indeed, 41 splits in $K^+$ and we write
$$ \q_1 = (-z^2 - 2z + 4), \quad  \q_2 = (-2z^2 + 3z + 4),  \quad \q_3 = (-3z^2 + z + 3).$$
We compute $B_{41}(g)$ for all $g \in \mathcal{L}$ and check it is always non-zero. Moreover, the primes that occur as divisors of all the $B_{13}(f)$ and $B_{41}(g)$ are 2,3,5. All the computations were done using Magma \cite{Magma}.

\subsection*{Newforms with non-rational coefficients.}
\label{subsec:nonrational}

To complete the proof of Theorem~\ref{grau7III} we have to eliminate the newforms with non-rational coefficients in the spaces $S_2(2^s 3 \pi_7)$ for $s=0,1,4$. 
As in the rational case we do this by obtaining restrictions for the exponent $p$. The bound we will obtain is not explicit due to the fact that we are not able to compute the full space when $s=4$.

\bigskip

We now explain how we deal with the newforms we have not computed. Let $f \in S_2(2^4 3 \pi_7)$ be a non-rational newform and suppose that \eqref{contradictionIII} holds for $f$. Let $q \neq 2,7$ be a rational prime and $\q \mid q$ a prime such that $a_\q(f) \not\in \Q$. Thus $\q$ is of good or multiplicative reduction for $E$ and we have, respectively,
$$a_{\q}(f) \equiv a_{\q}(E) \pmod{p} \quad \mbox{ or } \quad a_{\q}(f) \equiv \pm (\mbox{Norm}(\q) +1) \pmod{p}.$$
Let $P_\q(x)$ be the minimal polynomial of $a_{\q}(f)$. By applying $P_\q$ to both sides of the congruences we obtain
$$ 0 \equiv P_\q(a_{\q}(E)) \pmod{p} \quad \mbox{ or } \quad 0 \equiv P_\q(\pm (\mbox{Norm}(\q) +1)) \pmod{p}.$$ 
Since $P_\q$ is minimal, the right hand side of these congruences must be non-zero because $a_{\q}(E)$ and $\pm (\mbox{Norm}(\q) +1)$ are integers. From the Hasse-Weil bound we know there are only finitely many possibilities for $a_{\q}(E)$. Consequently, there are only finitely many primes dividing the right hand sides of the congruences above. Let $M_f$ be the product of all those primes. Hence, \eqref{contradictionIII} cannot hold for $f$ if $p \nmid M_f$. Since there are only finitely many newforms we can take $M_0$ to be the product of all $M_f$. Thus \eqref{contradictionIII} cannot hold for any non-rational $f$ in $S_2(2^4 3 \pi_7)$ if $p \nmid M_0$.

\bigskip

In the spaces corresponding to $s=0,1$ we computed the newforms explicitly. We eliminate all of them by making the previous argument concrete. We illustrate with an example. Let $\q_1 = (z^2 + z - 3)$ be a prime above 13. Let $g$ be the newform of level $2^1 3 \pi_7$ such that $a_{\q_1}(g)$ satisfies the polynomial $P_{\q_1}(x) = x^3 - 7x^2 + 10x + 7$. Suppose that \eqref{contradictionIII} holds for $g$. Hence we have
$$a_{\q_1}(g) \equiv a_{\q_1}(E) \pmod{p} \quad \mbox{ or } \quad a_{\q_1}(g) \equiv \pm (13+1) \pmod{p}.$$
We apply $P_{\q_1}$ to both sides of these congruence to obtain
$$ 0 \equiv P_{\q_1}(a_{\q_1}(E)) \pmod{p} \quad \mbox{ or } \quad 0 \equiv P_{\q_1}(\pm 14) \pmod{p}.$$
Since $a_{\q_1}(E) \in \{ -6, -2, 2 \}$ we get
$$ 0 \equiv -521,-49,7 \pmod{p} \quad \mbox{ or } \quad 0 \equiv -4249, 1519 \pmod{p},$$
thus we eliminate $g$ if $p \neq 7,31,521, 607$. We apply the same reasoning for all $g$ we computed. Theorem~\ref{grau7III} now follows if we pick $M_6$ be the product of $M_0$ by all the primes showing up when applying the argument above for all non-rational $g \in S_2(2^s 3 \pi_7)$ with $s=0,1$. 

\section{Equations of signature $(7,7,p)$, Part IV.}
\label{seteIV}

In this section, we focus on $C=3$. Due to the existence of trivial solutions it is clear that, using the techniques developed here, $C=3$ is the smallest value of $C$ such that one can hope to solve the equation $x^7 + y^7 = Cz^p$. By putting together all the results obtained so far with a further computation of Hilbert newforms we will now prove the following theorem.

\begin{tm} There is some constant $M_3$ such that if $p > (1 + 3^{18})^2$ and $p \nmid M_3$ then the equation $x^7 + y^7 = 3z^p$ has no non-trivial primitive solutions. 
\label{grau7IV}
\end{tm}

For the proof we will again use the Frey curves in section~\ref{sete} and in section~\ref{seteII}. Suppose $(a,b,c)$ is a primitive solution to $x^7 + y^7 = 3z^p$. Write $E:=E_{(a,b)}^{(1,2,3)}$ and $E^\prime := E_{(a,b)}^{(2,1)}$. 

\begin{tm} Suppose that $(a,b,c)$ is a primitive solution to $x^7 + y^7 =3z^p$. Write $I_3 = (1 + 3^{18})^2$. Then, the representations $\bar{\rho}_{E,p}$ and $\bar{\rho}_{E',p}$ are irreducible for all $p > I_3$.
\label{irr3}
\end{tm}
\begin{proof} We know from section~\ref{sete} that $\bar{\rho}_{E,p}$ is irreducible for all $p > 17$. The statement for $E'$ follows from the proof of Theorem~\ref{irr4}.
\end{proof}

From Propositions~\ref{artinc2I} and \ref{cond4}, modularity, Theorem~\ref{irr3} and level lowering we obtain, for large $p$, the isomorphisms: 
\begin{equation}
 \bar{\rho}_{E,p} \sim \bar{\rho}_{f,\fp}, 
 \label{eq:iso1}
\end{equation}
where $f$ is a classical newform $2^\ell 7^2$ with $\ell \in \{2,3,4\}$, and
\begin{equation}
\bar{\rho}_{E^\prime,p} \sim \bar{\rho}_{g,\fp^\prime}, 
\label{eq:iso2}
\end{equation}
where $g$ is a Hilbert newforms over $K^+$ of level $2^s 3 \pi_7^t$ with $s \in \{0,1,3,4\}$ and $t=1$ or $2$.

\bigskip

In what follows we will subdivide the solutions $(a,b,c)$ according to certain divisibility conditions at 2 and 7. Then, for each condition type, we will contradict \eqref{eq:iso1} for all $f$ in the corresponding level or \eqref{eq:iso2} for all $g$ in the corresponding level. Except for two condition types that will require new computations, all the other contradictions are obtained by applying the proofs of the theorems in previous sections. 

\bigskip

We first observe that since $C=3$ we cannot apply Theorem~\ref{grau7} as we did in section~\ref{seteIII}. Thus we also have to consider the case $7 \nmid a+b$. We now split $(a,b,c)$ into cases.
\begin{itemize}
 \item Suppose $2 \mid a+b$. Then $6 \mid a+b$ and the result follows from the proof of Theorem~\ref{grau7III}. 
 \item From now on assume $2 \nmid a+b$.
 \item From Propositions~\ref{conductor} and \ref{cond4} it follows that if $4 \nmid a$ then $\ell=3$ and $s=3,4$.
 \item Suppose $4 \nmid a$ and $7 \nmid a+b$. Then, using $E$, the result follows from the proof of Theorem~\ref{grau7}.
 \item Suppose $4 \nmid a$ and $7 \mid a+b$. Thus $t=1$. Using $E^\prime$, the result follows from the proof of Theorem~\ref{grau7III} when $s=4$. For $s=3$, using $E^\prime$, the result will follow if we can eliminate all the newforms in $S_2(2^3 3 \pi_7)$.
 \item From Proposition~\ref{cond4} we know that $s=3$ if $4 \mid a$. 
 \item Suppose $4 \mid a$ and $7 \mid a+b$. Thus $t=1$. Using $E^\prime$, the result will follow if we can eliminate all the newforms in $S_2(2^3 3 \pi_7)$.
 \item Suppose $4 \mid a$ and $7 \nmid a+b$. Thus $t=2$ and, using $E^\prime$, the result follows if we eliminate all the newforms in $S_2(2^3 3 \pi_7^2)$.
\end{itemize}

\begin{rem} Note that in the last two cases we cannot use $E$ instead of $E^\prime$. Indeed, if $2 \nmid a+b$ and $4 \mid a$ then isomorphism \eqref{eq:iso1} holds with $f \in S_2(2^2 7^2)$. The newform $f_{(1,0)}$ associated with $E_{(1,0)}$ belongs to this space. From the proof of Theorem~\ref{grau7} we know that $a_3(E) = a_3(E_{(1,0)}) = -1$ when $3 \mid a+b$. Thus we do not have enough information to eliminate $f_{(1,0)}$. Recall that in the mentioned proof we do not need to eliminate $f_{(1,0)}$ because we also have $2 \mid a+b$.
\end{rem}

To complete the proof we have to eliminate all the newforms in  $S_2(2^3 3 \pi_7^t)$ for $t=1,2$. For $t=1$ we used Magma to compute all the newforms and obtained 121 newforms where 56 are rational. We wrote a small Magma code to eliminate all these 56 newforms, analogously to what we did in section~\ref{subsec:rational}. We again used the primes above 13 and 41 and obtained a contradiction for all $p > 5$. Since we also know the non-rational newforms explicitly we eliminate them as in the last paragraph of section~\ref{seteIII}.

We now take care of $t=2$. The dimension of $S_2(2^3 3 \pi_7^2)$ is 10753 which makes a full computation impossible. Nevertheless, John Voight was also able to compute the rational newforms in it. There are 152 of them. More precisely, he computed their Fourier coefficients for all primes $\q$ with norm up to 200. The output can be found at \url{http://www.ub.edu/tn/visitant/amat.php}. As before, we used the primes above 13 and 41 and eliminated all the newforms for $p > 17$. Since we do not have the non-rational newforms explicitly we apply the same argument as in section~\ref{subsec:nonrational} which gives us the constant $M_3$. This ends the proof of Theorem~\ref{grau7IV}.

\bibliography{geralnovo}
\bibliographystyle{plain}

\end{document}